\documentclass{compositio} 
\usepackage[USenglish]{babel} 
\usepackage[T1]{fontenc}
\usepackage[ansinew]{inputenc}
 \usepackage{stmaryrd}
            
\usepackage[shortlabels]{enumitem}

\usepackage[toc,page]{appendix}
\usepackage{hyperref}

\usepackage{amsmath,amsthm,amsfonts,amssymb}%
\usepackage{bm}
\usepackage{graphicx}
\usepackage[all,cmtip]{xy}
\usepackage{tikz-cd}
\usepackage[nottoc]{tocbibind}
\tikzset{%
    symbol/.style={%
        ,draw=none
        ,every to/.append style={%
            edge node={node [sloped, allow upside down, auto=false]{$#1$}}}
    }
}
\usetikzlibrary{arrows}

\pgfdeclarelayer{background}
\pgfsetlayers{background,main}

\DeclareMathOperator {\Sym}{Sym}
\DeclareMathOperator {\Hom}{Hom}

\DeclareMathOperator {\Map}{Map}
\DeclareMathOperator {\Diff}{Diff}

\DeclareMathOperator {\id}{id}
\DeclareMathOperator {\Lie}{Lie}
\DeclareMathOperator {\Coh}{Coh}

\def \Z {\mathbb{Z}}
\def \C {\mathbb{C}}

\def \R {\mathbb{R}}
\def \bT {\mathbb{T}}

\def \D {\mathcal{D}}

\def \< {\langle}
\def \> {\rangle}

\def \S {\mathcal{S}}

\def \O {\mathcal{O}}

\def \N {\mathbb{N}}
\def \W {\mathcal{W}}

\def \E {\mathcal{E}}
\def \M {\mathcal{M}}
\def \F {\mathcal{F}}
\def \U {\mathcal{U}}

\def \H {\mathcal{H}}

\def \G {\mathcal{G}}

\def \X {\mathcal{X}}

\def \I  {\mathcal{I}}
\def \Ell {\mathcal{E}\ell\ell}

\def \to {\rightarrow}

\def \t {\mathfrak{t}}

\def \SL {\mathrm{SL}}
\def \GL {\mathrm{GL}}
\def \pt {\mathrm{pt}}

\theoremstyle{plain}
\newtheorem{theorem}{Theorem}[section]
\newtheorem{lemma}[theorem]{Lemma}
\newtheorem{proposition}[theorem]{Proposition}
\newtheorem{corollary}[theorem]{Corollary}

\theoremstyle{definition}
\newtheorem{definition}[theorem]{Definition}
\newtheorem{example}[theorem]{Example}
\newtheorem{notation}[theorem]{Notation}
\newtheorem{remark}[theorem]{Remark}
\newtheorem{conventions}[theorem]{Conventions}

\theoremstyle{remark}

\begin{document}
\title[Complex analytic elliptic cohomology from double loop spaces] 
{A construction of complex analytic elliptic cohomology from double free loop spaces} 
\author{Matthew Spong}
\email{matt.spong@gmail.com} 
\address{School of Mathematics and Statistics\\The University of Melbourne\\Parkville VIC 3010\\Australia}
\shortauthors{M. Spong}
\thanks{This research was in part supported by the Discovery Project grant DP160104912, Subtle Symmetries and the Refined Monster. } 
\begin{abstract}
We construct a complex analytic version of an equivariant cohomology theory which appeared in a paper of Rezk, and which is roughly modeled on the Borel-equivariant cohomology of the double free loop space. The construction is defined on finite, torus-equivariant CW complexes and takes values in coherent holomorphic sheaves over the moduli stack of complex elliptic curves. Our methods involve an inverse limit construction over all finite dimensional subcomplexes of the double free loop space, following an analogous construction of Kitchloo for single free loop spaces. We show that, for any given complex elliptic curve $\mathcal{C}$, the fiber of our construction over $\mathcal{C}$ is isomorphic to Grojnowski's equivariant elliptic cohomology theory associated to $\mathcal{C}$. 
\end{abstract}

\maketitle
\vspace{5mm}

The first construction of a complex analytic, equivariant elliptic cohomology theory was given by Grojnowski in 1994, in a note now published as \cite{Groj}. Although Grojnowski originally built his theory for the purpose of constructing certain elliptic algebras, it has since found numerous applications, having been used, for example, to give a conceptual proof of the rigidity of the Ochanine genus (see \cite{Rosu}). However, it is a technical and ad hoc construction, and alternate models have since been desired that clarify its relationship to fields outside of geometric representation theory. For some recent developments in this direction, see for example \cite{BET}, \cite{Kitch1}, and \cite{Rezk}.

In Section 5 of \cite{Rezk}, Rezk introduced a construction of a $G$-equivariant cohomology theory $E^*_G$ defined on $G$-CW complexes for a certain class of connected Lie groups $G$, which includes compact tori. It is not a complex analytic construction because the coefficient ring $E^*_G(\pt)$ is a polynomial ring rather than a ring of holomorphic functions. However, Rezk conjectured that if $E_G^*$ could be made complex analytic in a suitable way, then it would serve as a model for Grojnowski's theory. 

One of the main ingredients in Rezk's construction is the \textit{double free loop space} of a $G$-CW complex $X$, which we write as the space of continuous maps
\[
L^2X := \Map(\bT^2,X),
\]
where $\bT$ denotes the parametrised circle $\R/\Z$. The construction proceeds by considering the subspace of \textit{ghost maps}
\[
L^2X^{gh} \subset L^2X,
\]
which consists precisely of the maps $\bT^2 \to X$ whose image is contained in a single $G$-orbit. The space $L^2X$ comes equipped with a natural action of the semidirect product group
\[
\Diff(\bT^2) \ltimes \Map(\bT^2,G),
\]
which Rezk denotes by $\W(G)$, and which preserves $L^2X^{gh}$. There is a connected subgroup $\bT^2 \times G \subset \W(G)$ consisting of the group of translations $\bT^2 \subset \Diff(\bT^2)$ and the group of constant loops $G \subset L^2G$. Rezk defines $E_G^*(X)$ as the Borel $\bT^2 \times G$-equivariant cohomology ring
\begin{equation}\label{timid}
H^*_{\bT^2 \times G}(L^2X^{gh};\C),
\end{equation}
equipped with a natural action of $\C^\times \times \overline{\W}(G)$, where $\overline{\W}(G)$ is the discrete group $\pi_0 \W(G)$, and $\C^\times$ acts via the $\Z$-grading on cohomology.\footnote{We explain how a $\C^\times$-action relates to a $\Z$-grading in Remark \ref{modforms}, albeit in a somewhat different context.} Using the subspace of ghost maps $L^2X^{gh}$ instead of the full double loop space $L^2X$ ensures that $E_G^*$ is a cohomology theory in $X$. 

To see what $E_G^*$ has to do with elliptic curves, consider the case that $G$ is the trivial group $1$. Then \eqref{timid} is naturally a module over $H^*_{\bT^2}(\pt; \C)$, which may be identified with the ring $\C[t_1,t_2]$ of polynomial functions on 
\[
\Lie(\bT^2) \otimes \C \cong \C^2.
\]
Thus, by a standard trick, the module \eqref{timid} may be regarded as a sheaf of modules over $\C^2$. Rezk restricts the sheaf to the subspace $\X \subset \C^2$ consisting of those pairs of complex numbers which generate a lattice in $\C$, and calculates the induced action of
\[
\C^\times \times \overline{\W}(1) \cong \C^\times \times \GL_2(\Z)
\]
on $\X$. Rezk observes that this action classifies complex elliptic curves, in the sense that the quotient stack associated to the action is the moduli stack $\M$ of complex elliptic curves. 

Rezk's conjecture is that a suitably defined complex analytic version of $E^*_G(X)$ would, for a finite $G$-CW complex $X$, yield a coherent holomorphic sheaf over $\M$, the fiber of which is Grojnowski's theory for a particular elliptic curve. However, as Rezk points out, tensoring $E^*_G(X)$ with holomorphic coefficients does not behave well, because $E^*_G(X)$ is often non-Noetherian, even when $X$ is a $G$-orbit. In this article, we solve this problem in the case of a torus $G = T$ by applying an idea that appeared in a paper \cite{Kitch1} of Kitchloo.\footnote{Although Kitchloo applied this idea in the context of the K-theory of single free loop spaces, our construction is largely analogous to his, and uses formal properties common to both Borel equivariant cohomology and equivariant K-theory.} Namely, the idea is to replace the cohomology ring 
\[
H^*_{\bT^2 \times T}(L^2X^{gh};\C)
\]
with the inverse limit of sheaves
\[
\H^*_{\bT^2 \times T}(L^2X) := \varprojlim_{Y \subset L^2X}  H^*_{\bT^2 \times T}(Y;\C) \otimes_{H_{\bT^2 \times T}(\pt;\C)} \O_{\X^+ \times \Lie(T) \otimes \C}
\]
over all finite subcomplexes $Y$ of $L^2X$. Here we view the cohomology ring with the $\Z/2\Z$-grading by even and odd degree, while the sheaf $\O_{\X^+ \times \Lie(T)\otimes \C}$ of holomorphic functions has the trivial grading. Tensoring with holomorphic functions before applying the limit behaves well because the Borel equivariant cohomology ring of a finite CW-complex is finitely generated over $H_{\bT^2 \times T}(\pt)$. To make computations more tractable, we will show that the inverse limit above is isomorphic to the inverse limit over a much smaller set $\D(X)$ of subspaces of $L^2X$, the colimit of which is a subspace of $L^2X^{gh}$. Thus, our construction only depends on ghost maps, which supports the view that it yields a complex analytic version of Rezk's construction. 

It turns out that the inverse limit sheaf $\H^*_{\bT^2 \times T}(L^2X)$ admits an interesting action of $\C^\times \times W_{\widetilde{L^2T}}$, where $W_{\widetilde{L^2T}}$ is the Weyl group of the maximal torus $\bT^2 \times T$ in the extended double loop group\footnote{Here, $W_{\widetilde{L^2T}}$ plays the role of Rezk's group $\overline{\W}(T)$, and $\widetilde{L^2T}$ plays the role of $\W(T)$.} 
\[
\widetilde{L^2T} := (\SL_2(\Z) \ltimes \bT^2) \ltimes L^2T.
\]
We show that taking invariants of $\H^*_{\bT^2 \times T}(L^2X)$ yields a coherent, $\Z/2\Z$-graded holomorphic sheaf $\E_T^*(X)$ over a certain stack $\M_T$ over $\M$, and that the fiber of $\E_T^*$ over any given elliptic curve is naturally isomorphic to Grojnowski's cohomology theory. This confirms Rezk's conjecture.

We now give an outline of the structure of the paper. In Section \ref{bread} we introduce some basic objects, including the moduli stack
\[
\M = [(\C^\times \times \SL_2(\Z)) \backslash\!\!\backslash \X^+]
\]
of complex elliptic curves. In our case, $\M$ is modelled on the $\C^\times \times \SL_2(\Z)$-equivariant space $\X^+$ of pairs of complex numbers $(t_1,t_2)$ such that $t_1/t_2$ has positive imaginary part.\footnote{This is equivalent to the quotient stack obtained via the action of $\C^\times \times \GL_2(\Z)$ on $\X$.} In Section \ref{milk}, we introduce Borel-equivariant cohomology and state some of its important properties, and then in Section \ref{Grojn} we summarise Grojnowski's construction. In Section \ref{water}, we introduce the $\C^\times \times \SL_2(\Z)$-equivariant space $E_T$, which will provide us with a model for the stack $\M_T$. We show in Section \ref{soup} that, given a finite $T$-CW complex $X$, there exists an open cover of $E_T$ which is adapted to $X$ in a certain sense. The open cover will be used in Section \ref{low} to obtain some fixed-point results for local values of the inverse limit sheaf $\H^*_{\bT^2 \times T}(L^2X)$, defined in Definition \ref{varlimm}. We use those fixed-point results to give a computable description of $\H^*_{\bT^2 \times T}(L^2X)$ in Theorem \ref{yess}, and then again in Theorem \ref{actyon1} to show that $\H^*_{\bT^2 \times T}(L^2X)$ admits an action of $\C^\times \times W_{\widetilde{L^2T}}$. Section \ref{brew} culminates with the main construction in this paper, which is the $\C^\times \times \SL_2(\Z)$-equivariant sheaf $\E_T^*(X)$ on $E_T$ defined in Definition \ref{sherpa}. This is produced essentially as the invariants of the sheaf $\H^*_{\bT^2 \times T}(L^2X)$ with respect to a lattice subgroup of $\C^\times \times W_{\widetilde{L^2T}}$. The equivariant sheaf $\E_T^*(X)$ is equivalent to a sheaf on the quotient stack 
\[
\M_T = [(\C^\times \times \SL_2(\Z)) \backslash\!\!\backslash E_T] 
\]
over $\M$. In Section \ref{salt}, we show that the fiber of $\E^*_T$ over any elliptic curve $\mathcal{C} \in \M$ is a $T$-equivariant elliptic cohomology theory, in a sense appropriate to Grojnowski's construction. Section \ref{pepper} is a calculation of $\E^*_T(X)$ for a $T$-orbit $X = T/K$, which allows one to compute $\E_T^*(X)$ for any finite complex $X$ using the Mayer-Vietoris sequence. The final section is Section \ref{calf}, in which we give a local description of the fiber of $\E^*_{T}(X)$ over an elliptic curve $\mathcal{C}$. In other words, we fix an arbitrary curve $\mathcal{C}$ and an open cover adapted to a finite complex $X$, we take the restriction of $\E^*_T(X)$ to $\mathcal{C} \in \M$, and we compute it as a collection of sheaves indexed by the elements of the open cover, along with some gluing maps. Our main result is that this is exactly Grojnowski's $T$-equivariant elliptic cohomology of $X$, for the elliptic curve $\mathcal{C}$. This is Corollary \ref{glutes}.

\begin{conventions}
All group actions are from the left, unless otherwise indicated. For a group $G$ acting on a space $X$, we use $g \cdot x$ to denote the action of $g\in G$ on $x \in X$, and we use $gg'$ to denote the group product of $g,g' \in G$. We denote by $G\backslash\!\!\backslash X$ the corresponding action groupoid, and we enclose this within square brackets to denote the underlying quotient stack. A sheaf on the action groupoid $G\backslash\!\!\backslash X$ is equivalent to a $G$-equivariant sheaf on $X$, which yields a sheaf on the quotient stack $[G\backslash\!\!\backslash X]$ (as in \cite[\href{https://stacks.math.columbia.edu/tag/06WT}{Tag 06WT}]{stacks}). 

If $X$ and $Y$ are topological spaces, then the set of continuous maps $\Map(X,Y)$ is regarded as a space with the compact-open topology. If $A$ is an abelian group, $H$ an arbitrary group, and $A$ acts on $H$, then our convention for the group law of the semidirect product $A \ltimes H$ is
\[
(a',h')(a,h) = (a'a, a^{-1}\cdot h' h).
\]
The tensor product of two $\Z$-modules is over $\Z$, unless otherwise specified. All rings are assumed to have a multiplicative identity. For (not necessarily square) matrices $A$, $m$, and $t$, we use expressions such as $Am$, $mt$, and $mA$ to mean matrix multiplication. So, for example, if $m = (m_1,m_2)$ and $t = (t_1,t_2)$ are vectors, then $mt$ means the dot product, where the transpose of a vector should be understood wherever it is necessary to make sense of an expression.
\end{conventions}

\begin{acknowledgements}
This paper originated in my doctoral thesis and owes its existence to my supervisor, Nora Ganter, who has been an invaluable source of encouragement and support. Most of the ideas in this paper are due to Charles Rezk and Nitu Kitchloo, and I am grateful to them for reading my thesis and for their valuable comments. I also wish to thank Daniel Berwick-Evans, Christian Haesemeyer, Arun Ram, Marcy Robertson, Yaping Yang and Gufang Zhao for stimulating conversations. Finally, I am indebted to the anonymous referee for closely reading the original manuscript and providing many valuable suggestions and corrections.
\end{acknowledgements}

\section{Elliptic curves over $\C$ and other basic objects}\label{bread}

In this section, we list some well known facts concerning the classification of elliptic curves over $\C$, drawn from the short summary appearing in Section 2 of \cite{Rezk}. We also introduce some other basic objects, including the base space $E_{T,t}$ of Grojnowski's construction \cite{Groj}.

\begin{remark}
Consider the subspace 
\[
\X := \{ (t_1,t_2) \in \C^2 \, | \, \R t_1 + \R t_2 = \C \} \subset \C^2.
\]
An element $t = (t_1,t_2) \in \X$ defines a lattice
\[
\Lambda_t := \Z t_1 + \Z t_2 \subset \C.
\]
It is easily verified that $\X$ is preserved under left multiplication by $\GL_2(\Z)$, and that $\Lambda_t = \Lambda_{t'}$ if and only if there is a matrix $A \in \GL_2(\Z)$ such that $At = t'$. 
\end{remark}

\begin{definition}
An \textit{elliptic curve over $\C$} is a complex manifold
\[
E_t := \Lambda_t \backslash \C,
\]
along with the quotient group structure induced by the additive group $\C$. A \textit{map of elliptic curves} $E_t \to E_{t'}$ is induced by multiplication by a nonzero complex number $\lambda$ satisfying $\lambda \Lambda_t \subset \Lambda_{t'}$. Such a map is an isomorphism if and only if $\lambda \Lambda_t = \Lambda_{t'}$. 
\end{definition}

\begin{remark}\label{dominic}
Two elliptic curves $E_t$ and $E_{t'}$ are equal if and only if there exists a matrix $A \in \GL_2(\Z)$ such that $At = t'$, and isomorphisms $E_t \cong E_{t'}$ correspond bijectively to pairs $(\lambda,A) \in \C^\times \times \GL_2(\Z)$ such that $\lambda A t = t'$. Therefore, elliptic curves over $\C$ are classified by the action of $\C^\times \times \GL_2(\Z)$ on $\X$ given by $(\lambda,A)\cdot t = \lambda At$. Alternatively, they are also classified by the action of the subgroup 
\[
\C^\times\times \SL_2(\Z) \subset \C^\times \times \GL_2(\Z)
\]
on the subspace
\[
\X^+ = \{(t_1,t_2) \in \X \, |\, \mathrm{Im}(t_1/t_2) > 0\} \subset \X,
\]
which is easily seen to inherit such an action. The corresponding moduli stack of complex elliptic curves is written
\[
\M := [(\C^\times\times \SL_2(\Z)) \backslash\!\!\backslash \X^+].
\]
It will in fact be necessary for us to instead model $\M$ on the $\C^\times \times \SL_2(\Z)$-action on $\X^+$ given by $(\lambda,A)\cdot t = \lambda^2 At$. However, this is easily seen to be equivalent to the above definition of $\M$.
\end{remark}

\begin{remark}\label{fad}
In this paper, we use $K$ to denote an arbitrary compact abelian group, we use $T$ to denote an arbitrary compact torus, and we write $\bT$ for the parameterised circle $\R/\Z$. We define the cocharacter lattice of a compact abelian group $K$ to be the group of continuous group homomorphisms
\[
\check{K} := \Hom(\bT,K).
\]
The evaluation map $\check{K} \otimes \bT \to K$ is an isomorphism onto the identity component of $K$, which induces a canonical identification
\[
\check{K} \otimes \R \cong \Lie(K)
\]
of Lie algebras. Similarly, the exponential map $\exp_K$ of $K$ is canonically identified with the composite map
\[
\check{K} \otimes \R \twoheadrightarrow \check{K} \otimes \bT \hookrightarrow K
\] 
where the first map is induced by the projection $\exp: \R \twoheadrightarrow \R/\Z = \bT$. Note that the kernel of $\exp_K$ is $\check{K} \otimes \Z \cong \check{K}$. We write $\t$ for the Lie algebra of a torus $T$ and $\t_\C$ for its complexification $\t \otimes_\R \C$, since we use these objects often. Similarly, we write $\Lie(K)_\C$ for $\Lie(K) \otimes_\R \C$.
 
Moreover, let $\hat{K}$ denote the character group $\Hom(K,\bT)$ of $K$. By the Pontryagin duality theorem, we have a canonical isomorphism 
\[
K \cong \Hom(\hat{K},\bT)
\]
of compact abelian groups. There is a canonical homomorphism $\check{K} \to \Hom(\hat{K},\Z)$, natural in $K$, which arises from the pairing 
\[
\begin{array}{rcl}
\hat{K} \times \check{K} & \longrightarrow & \Hom(\bT,\bT) \cong \Z \\
(\mu,m) & \longmapsto & \mu \circ m,
\end{array}
\]
where the identity map in $\Hom(\bT,\bT)$ is identified with $1 \in \Z$. This homomorphism induces a commutative diagram
\begin{equation}\label{psy}
\begin{tikzcd}
\Lie(K) \cong \check{K} \otimes \R \ar[r,"\sim"] \ar[d, twoheadrightarrow] & \Hom(\hat{K},\R)\ar[d, "\exp_K"]  \\
\check{K} \otimes \bT \ar[r, hookrightarrow] & \Hom(\hat{K},\bT) \cong K
\end{tikzcd}
\end{equation}
where the vertical maps are induced by the projection $\exp: \R \twoheadrightarrow \R/\Z = \bT$. Since $\R$ is a free $\Z$-module, the upper horizontal map is an isomorphism. The lower horizontal map is canonically identified with the inclusion of the identity component into $K$, so that we may identify the right vertical map with the exponential map $\exp_K$ of $K$. In the case of a torus $K = T$, the pairing is perfect, so that we have a canonical identification $\check{T} \cong \Hom(\hat{T},\Z)$. Therefore, for a torus, the lower horizontal map of diagram \eqref{psy} is an isomorphism.

For the elliptic curve $E_t := \Lambda_t \backslash \C$ corresponding to $t \in \X$, define
\[
E_{K,t} := \Hom(\hat{K},E_t).
\]
In the case of a torus $K = T$, there is a canonical isomorphism
\[
E_{T,t} \cong \check{T} \otimes E_t
\]
and the quotient map $\C \twoheadrightarrow E_t$ induces a quotient map
\[
\zeta_{T,t}: \t_\C \twoheadrightarrow E_{T,t},
\]
by tensoring with $\check{T}$.
\end{remark}

\begin{definition}
Let $T$ be a compact torus. A $T$-CW complex $X$ is defined as a union 
\[
\bigcup_{n\in \N} X^n
\]
of $T$-subspaces $X^n$ such that
\begin{enumerate}
\item $X^0$ is a disjoint union of orbits $T/K$, where $K \subset T$ is a closed subgroup; and
\item $X^{n+1}$ is obtained from $X^n$ by attaching $T$-cells $T/K \times D^{n+1}$ along $T$-equivariant attaching maps $T/K \times S^n \rightarrow X^n$, where $T$ acts trivially on $D^{n+1}$.
\end{enumerate}
A \textit{finite $T$-CW complex} is a $T$-CW complex which is a union of finitely many $T$-cells. A \textit{pointed $T$-CW complex} is a $T$-CW complex along with a distinguished $T$-fixed basepoint in the $0$-skeleton of $X$. A map $f: X \to Y$ of (pointed) $T$-CW complexes is a $T$-equivariant map such that $f(X^n) \subset Y^n$ for all $n$ (and preserving the basepoint). 
\end{definition}

\begin{example}\label{repsphere}
Let $T$ be a rank one torus and let $\lambda \in \hat{T}$ be an irreducible character of $T$. The \textit{representation sphere associated to $\lambda$} is the one-point compactification $S_\lambda$ of the one dimensional complex representation $\C_\lambda$ associated to $\lambda$. This may be equipped with the structure of a finite $T$-CW complex where
\[
X^0 = T/T \times \{\infty\} \: \amalg \: T/T \times \{0\},
\]
and $X^1 =  T \times D^1$, with $T$-equivariant attaching map $T \times S^0 \rightarrow X^0$ given by sending one end of $D^1$ to $\{0\}$ and the other end to $\{\infty\}$. An element $z \in T$ acts by multiplication by $\lambda(z)$ on the left factor of $X^1$ and trivially on the right factor.
\end{example}

\section{Some properties of Borel-equivariant cohomology}\label{milk}

In this section we introduce Borel-equivariant cohomology, a fundamental ingredient of our construction, and state several of its properties which will be useful to us. First note that, for any topological group $G$, one obtains by the Milnor construction a contractible space $EG$ with a free right action of $G$. For a $G$-space $X$, the Borel construction $EG \times_G X$ of $X$ is the topological quotient of $EG \times X$ by the equivalence relation $(x\cdot g,y) \sim (x,g\cdot y)$. The Borel-equivariant cohomology of $X$ is then defined as the cohomology ring $H^*(EG \times_G X; \C)$. We use the notation $H_G^*(X)$ to mean the $\Z/2\Z$-graded commutative ring 
\[
H^*_G(X) = H^{even}(EG \times_G X; \C) \oplus H^{odd}(EG \times_G X; \C) 
\]
graded by the parity of the cohomological degree. We write $H^*_G$ for $H^*_G(\pt) = H^*(BG;\C)$, and will often drop the asterisk from this ring since $H^{odd}_G(\pt) = 0$ (see Remark \ref{remmie}). Since the unique map from $X$ to a point induces a map $H^*_G \to H^*_G(X)$ of graded rings, $H^*_G(X)$ is naturally a $\Z/2\Z$-graded commutative algebra over $H^*_G$. 

For the remainder of this section, we return to the case where $G$ is a compact torus $T$. Our reference for Borel-equivariant cohomology is \cite{AB}.

\begin{proposition}\label{change1}
There is an isomorphism of $\Z/2\Z$-graded rings
\[
H^*_T(T/K) \cong H^*_K.
\]
\end{proposition}

\begin{proof}
Since $K$ acts freely on $ET$, the space $ET$ is a model for $EK$. Therefore, 
\[
ET \times_T T/K \cong ET/K
\]
is a model for $BK$. 
\end{proof}

In the next two lemmas, we consider the following setup. Let $T \twoheadrightarrow K \twoheadrightarrow G$ be a composition of surjective maps of compact abelian groups, where $T$ is a torus. If $X$ is a finite $G$-CW complex, then these maps induce a commutative diagram 
\begin{equation}
\begin{tikzcd}
ET \times_T X \ar[d,"h"] \ar[r,"j"] & BT \ar[d,"p"] \\
EK \times_K X \ar[r,"g"] \ar[d,"f"] & BK \ar[d] \\
EG \times_G X \ar[r] & BG 
\end{tikzcd}
\end{equation}
where both squares are pullback diagrams. The following lemma is proved as Proposition 2.3.4 in Chen's thesis \cite{Chen}. 

\begin{proposition}\label{changeh}
Let $X$ be a finite $K$-CW complex. There is an isomorphism of $\Z/2\Z$-graded $H_T$-algebras
\[
H^*_{K}(X) \otimes_{H_{K}} H_T \cong H^*_T(X)
\]
natural in $X$, and induced by $h^* \cup j^*$. 
\end{proposition}

\begin{remark}
By Proposition \ref{changeh}, we have an induced diagram of isomorphisms of graded $H_T$-algebras  
\begin{equation}\label{pullback2}
\begin{tikzcd}
H^*_G(X) \otimes_{H_G} H_K \otimes_{H_K} H_T \,\, \ar[d] \ar[r,"{f^* \cup g^* \otimes \id}"] & \, \, H^*_K(X) \otimes_{H_K} H_T \ar[d,"{h^* \cup j^*}"] \\
H^*_G(X) \otimes_{H_G} H_T \ar[r,"{(f\circ h)^* \cup j^*}" ] & H_T(X)
\end{tikzcd}
\end{equation}
where the left vertical map is the canonical map induced by $p^*: H_K \to H_T$.
\end{remark}

\begin{lemma}\label{kurt}
The diagram \eqref{pullback2} commutes.
\end{lemma}

\begin{proof}
It suffices to show that diagram \eqref{pullback2} commutes for an element of the form $a \otimes b \otimes c$. We have
\[
\begin{tikzcd}
a \otimes b \otimes c \ar[r,mapsto] \ar[d,mapsto]& (f^*a \cup g^*b) \otimes c \ar[d,mapsto] \\
a \otimes (p^*b \cup c) \ar[r,mapsto] & (f\circ h)^*a \cup j^*(p^*b \cup c) = h^*(f^*a \cup g^*b) \cup j^*c,
\end{tikzcd}
\]
where equality holds since 
\[
\begin{array}{rcl}
(f\circ h)^*a \cup j^*(p^*b \cup c) &=& (f\circ h)^*a \cup (p\circ j)^*b \cup j^*c \\
&=& (f\circ h)^*a \cup (g\circ h)^*b \cup j^*c \\
&=& h^*f^*a \cup h^*g^*b \cup j^*c \\
&=& h^*(f^*a \cup g^*b) \cup j^*c.
\end{array}
\]
\end{proof}

\begin{remark}\label{remmie}
Let $\t_\C^\vee$ denote the dual vector space of $\t_\C$. The Chern-Weil isomorphism
\[
\Sym^*(\t_\C^\vee) \cong H^{2*}(BT;\C) 
\]
identifies the ring $H_T$ with the ring of polynomial functions on $\t_\C$. Note that the generators of the polynomial ring correspond to cohomology classes of degree $2$. The map is produced as follows. Since $T$ is a torus, there is an identification 
\[
\hat{T} \otimes \C \cong \Hom(\check{T},\Z) \otimes \C \cong \t_\C^\vee.
\]
Let $\C_\lambda$ be the representation corresponding to an irreducible character $\lambda \in \hat{T}$. The map  
\[
\lambda \mapsto c_1(ET \times_T \C_\lambda)  
\]
induces an isomorphism $\hat{T} \cong H^2(BT;\Z)$, where $c_1$ denotes the first Chern class. Tensoring this map with $\C$ and extending by the symmetric product yields the isomorphism $\Sym^*(\t_\C^\vee) \cong H^*(BT;\C)$. See Proposition 2.6 in \cite{Rosu03} for details.
\end{remark}

\begin{definition}\label{bochum}
Let $T$ be a torus, let $X$ be a finite $T$-CW complex, and let $\O_{\t_\C}$ be the holomorphic structure sheaf of the complex manifold $\t_\C$. We denote by $\H^*_T(X)$ the $\Z/2\Z$-graded holomorphic sheaf of $\O_{\t_\C}$-algebras whose value on an analytic open set $U \subset \t_\C$ is
\[
H^*_T(X) \otimes_{H_T} \O_{\t_\C}(U).
\]
The tensor product is over the map $H_T \to \O_{\t_\C}(U)$ which identifies an element of $H_T$ with its corresponding polynomial, considered as a function on $U \subset \t_\C$. Equipped with the restriction maps of $\O_{\t_\C}$, it follows from Propositions 2.8 and 2.10 in \cite{Rosu03} that $\H^*_T(X)$ is a sheaf, and not just a presheaf. We write $\H^*_T(X)_V$ for the sheaf obtained by restricting $\H^*_T(X)$ to a subset $V \subset \t_\C$. 
\end{definition}

\begin{remark}\label{ainur}
Although it will remain somewhat in the background until we reach Section \ref{brew}, we mention at this point that $\H^*_T(X)$ may be equipped with a $\C^\times$-equivariant structure such that the $\Z/2\Z$-grading on $\H^*_T(X)$ corresponds to the eigenspaces of $\{\pm 1\} \subset \C^\times$. Consider the action map $\alpha: \C^\times \times \t_\C \to \t_\C$ given by $\alpha(\lambda, x) = \lambda^2 x$, and the projection $\pi: \C^\times \times \t_\C \twoheadrightarrow \t_\C$. Let $\pi_{23}: \C^\times \times \C^\times \times \t_\C \twoheadrightarrow \C^\times \times \t_\C$ be the projection along the first factor, and let $\mu:\C^\times \times \C^\times \to \C^\times$ be the group multiplication map. An equivariant structure on $\H^*_T(X)$ is the datum of an isomorphism $I: \alpha^* \H^*_T(X) \to \pi^*\H_T(X)$ which satisfies the cocycle condition 
\[
(\pi_{23})^*I \circ (\id_{\C^\times} \times \alpha)^*I =  (\mu \times \id_{\t_\C})^*I.
\]
Then, for a class $c \in H^n(ET \times_T X;\C)$ of degree $n \in \Z$, an open set $U \subset \C^\times \times \t_\C$, and holomorphic functions $f \in \O_{\t_\C}(\alpha(U))$ and $h \in \O_{\C^\times \times \t_\C}(U)$, we define the isomorphism $I$ by
\[
\begin{array}{rcl}
\H^*_T(X)(\alpha(U)) \otimes_{\O_{\t_\C}(\alpha(U))} \O_{\C^\times \times \t_\C}(U) &\rightarrow& \H^*_T(X)(\pi(U)) \otimes_{\O_{\t_\C}(\pi(U))} \O_{\C^\times \times \t_\C}(U) \\
c \otimes f \otimes h &\mapsto & c \otimes 1 \otimes (\alpha^*f) h \lambda^n
\end{array}
\]
where $\lambda$ is the composite map $U \subset \C^\times \times \t_\C \twoheadrightarrow \C^\times$. Note that if $c \in H^{2n}(BT;\C)$ corresponds via the Chern-Weil isomorphism to a homogeneous polynomial $p_c$ of degree $n$, then $c \otimes 1 \otimes 1 = 1 \otimes p_c \otimes 1$. So, we check that $I$ is well-defined as follows
\[
\begin{array}{rcl}
I(c \otimes 1 \otimes 1) &=& c \otimes 1 \otimes \lambda^{2n} \\
&=& 1 \otimes 1 \otimes (\pi^* p_c) \lambda^{2n}\\
&=& 1 \otimes 1 \otimes \alpha^* p_c \\
&=& I(1 \otimes p_c \otimes 1),
\end{array}
\]
where the third equality holds by definition of $\alpha$. Finally, it is straightforward to check that this definition of $I$ satisfies the cocycle condition, and we leave this to the reader.
\end{remark}

\begin{proposition}\label{Segall}
Let 
\[
\pi : T \twoheadrightarrow K
\]
be a surjective map of compact abelian groups, and let $X$ be a $K$-CW complex. The natural map
\[
(\Lie(\pi)_\C)^*\, \H^*_{K}(X) \longrightarrow  \H^*_T(X)
\]
is an isomorphism of $\Z/2\Z$-graded $\O_{\t_\C}$-algebras.
\end{proposition}

\begin{proof}
This follows immediately from Proposition \ref{changeh} by extending to holomorphic sheaves.
\end{proof}

\begin{definition}
Let $x \in \t_\C$. The inclusion of a closed subgroup $K \subset T$ induces an inclusion of complex Lie algebras $\mathrm{Lie}(K)_\C \subset \t_\C$. Let $T(x)$ denote the intersection
\[
\bigcap_{x \in \mathrm{Lie}(K)_\C} K
\]
of all closed subgroups $K \subset T$ whose complexification contains $x$. For a finite $T$-CW complex $X$, denote by $X^{x}$ the subspace of points fixed by $T(x)$.
\end{definition}

We now state a well known fixed-point theorem for Borel-equivariant cohomology, for which a proof may be found at Theorem 2.2.18 in \cite{Spong}. Note that, if $\F$ is a sheaf on a space $S$ and $x$ is a point in $S$, then we write $\F_x$ for the stalk of $\F$ at $x$. Thus, the symbol $\O_{\t_\C,x}$ denotes the ring of germs of holomorphic functions at $x$.

\begin{theorem}\label{localisationn}
Let $x \in \t_\C$ and $X$ be a finite $T$-CW complex. The restriction along $X^x \hookrightarrow X$ induces an isomorphism
\[
\H^*_T(X)_x \longrightarrow \H^*_T(X^x)_x
\]
of $\Z/2\Z$-graded $\O_{\t_\C,x}$-algebras.
\end{theorem}

\section{Grojnowski's equivariant elliptic cohomology}\label{Grojn}

There are already many accounts of the construction of Grojnowski's equivariant elliptic cohomology theory (e.g. \cite{Ando}, \cite{Chen}, \cite{Nora}, \cite{GKV}, \cite{Groj}, and \cite{Rosu}). Nonetheless, we sketch a brief description of the torus-equivariant version here because it is important for our main results. In this section, we fix an elliptic curve $E_t = \C/\Lambda_t$. Recall the quotient map 
\[
\zeta_{T,t}: \t_\C \twoheadrightarrow E_{T,t} \cong \check{T} \otimes E_t
\]
of Remark \ref{fad}. In this paper, we will use the following definition of a (reduced) $T$-equivariant elliptic cohomology theory.

\begin{definition}\label{tear}
A \textit{reduced $T$-equivariant elliptic cohomology theory associated to $E_t$} consists of data:
\begin{enumerate}
\item a contravariant functor $\F^*_{T,t}$ from the category of pointed finite $T$-CW complexes into the category $\Coh(E_{T,t})$ of $\Z/2\Z$-graded, coherent $\O_{E_{T,t}}$-algebras; and
\item \textit{Suspension isomorphism:} an isomorphism $\F^{*+1}_{T,t}(S^1\wedge X) \cong \F^{*}_{T,t}(X)$, natural in $X$. 
\end{enumerate}
The functor $\F^*_{T,t}$ must satisfy the following conditions:
\begin{enumerate}
\item \textit{Homotopy invariance:} two $T$-homotopic maps induce the same maps on $\F^*_{T,t}$;
\item \textit{Exactness:} applied to a cofiber sequence of finite $T$-CW complexes, $\F^*_{T,t}$ yields an exact sequence in $\Coh(E_{T,t})$; and
\item \textit{Additivity:} for $X$ equal to the wedge sum of spaces $X_i$ over all $i \in A$, the canonical map $\F^*_{T,t}(X) \to \prod_{i \in A} \F_{T,t}^*(X_i)$ is an isomorphism of $\Z/2\Z$-graded, coherent $\O_{E_{T,t}}$-algebras.
\end{enumerate}
\end{definition}

At the end of this section, we show that the reduced version of Grojnowski's theory $\widetilde{\G}_{T,t}^*$ satisfies Definition \ref{tear}.

\begin{definition}\label{teed}
Let $a \in E_{T,t}$. Define $T(a)$ as the intersection
\[
T(a) = \bigcap_{a \in E_{K,t}} K
\]
of closed subgroups $K \subset T$. For a $T$-CW complex $X$, denote by $X^a$ the subspace of points fixed by $T(a)$.
\end{definition}

\begin{remark}
If $\S$ is a finite set of closed subgroups of $T$, we can define an ordering on the points of $E_{T,t}$ by saying that $a \leq_\S b$ if $b \in E_{K,t}$ implies $a \in E_{K,t}$, for any $K \in \S$. If $\S$ is understood, then we just write $\leq$ for $\leq_\S$. 
\end{remark}

\begin{definition}\label{ruff}
If $X$ is a finite $T$-CW complex, let $\S(X)$ be the finite set of isotropy subgroups of $X$. If $f: X \to Y$ is a map of finite $T$-CW complexes, let $\S(f)$ be the finite set of isotropy subgroups which occur in either $X$ or $Y$. An open set $U$ in $E_{T,t}$ is \textit{small} if $\zeta_{T,t}^{-1}(U)$ is a disjoint union of connected components $V$ such that $V \cong U$ via $\zeta_{T,t}$.
\end{definition}

\begin{definition}\label{ross}
Let $\S$ be a finite set of closed subgroups of $T$. An open cover 
\[
\U = \{U_a\}_{a \in E_{T,t}}
\]
of $E_{T,t}$, which is indexed by the points of $E_{T,t}$, is said to be \textit{adapted to $\S$} if it has the following properties:
\begin{enumerate}
\item $a \in U_{a}$, and $U_{a}$ is small.
\item If $U_{a} \cap U_{b} \neq \emptyset$, then either $a \leq_\S b$ or $b \leq_\S a$.
\item If $a \leq_\S b$ and there exists $K \in \S$ such that $a \in E_{K,t}$ and $b \notin E_{K,t}$, then $U_b \cap E_{K,t} = \emptyset$.
\item Let $a$ and $b$ lie in $E_{K,t}$ for some $K \in \S$. If $U_a \cap U_b \neq \emptyset$, then $a$ and $b$ belong to the same connected component of $E_{K,t}$.
\end{enumerate}
\end{definition}

The following result is Theorem 2.2.8 in \cite{Chen}.

\begin{lemma}\label{constru}
For any finite set $\S$ of subgroups of $T$, there exists an open cover $\U$ of $E_{T,t}$ adapted to $\S$. Any refinement of $\U$ is also adapted to $\S$.
\end{lemma}

\begin{notation}
For $a \in E_{T,t}$ let 
\[
\mathrm{tr}_a: E_{T,t} \longrightarrow E_{T,t}
\]
denote translation by $a$. 
\end{notation}

\begin{remark}\label{canon}
Let $X$ be a finite $T$-CW complex and let $\U$ be a cover of $E_{T,t}$ which is adapted to $\S(X)$. Let $x \in \zeta_{T,t}^{-1}(a)$, and let $V_{x}$ be the component of $\zeta_{T,t}^{-1}(U_a)$ containing $x$. Let $V \subset V_x$ and $U \subset U_a$ be open subsets such that $V \cong U$ via $\zeta_{T,t}$. Since $U_a \in \U$ is small by the first property of an adapted cover, the map $\zeta_{T,t}$ induces an isomorphism of complex analytic spaces $V - x \cong U-a$. We may therefore consider the composite ring map
\begin{equation}\label{ring}
H_T \hookrightarrow \O_{\t_\C}(V-x) \cong \O_{E_{T,t}}(U - a).
\end{equation}
\end{remark}

\begin{definition}
Let $X$ be a finite $T$-CW complex. For each $U_a \in \U(X)$, define a sheaf $\G^*_{T,t}(X)_{U_a}$ of $\Z/2\Z$-graded $\O_{U_a}$-algebras which takes the value
\[
H^*_T(X^a) \otimes_{H_T} \O_{E_{T,t}}(U-a),
\]
on $U \subset U_a$ open, with restriction maps given by restriction of holomorphic functions. The tensor product is defined over \eqref{ring}, and the $\O_{U_a}$-algebra structure is given by multiplication by $t_a^* f$ for $f \in \O_{U_a}(U)$. The grading is induced by the odd and even grading on the cohomology ring.
\end{definition}

\begin{remark}
For a finite $T$-CW complex $X$, we have defined a sheaf on each patch $U_a$ of a cover $\U$ adapted to $\S(X)$. The next task is to glue the local sheaves together on nonempty intersections $U_a \cap U_b$ in a compatible way. To do this, we need to define gluing maps 
\[
\phi_{b,a}: \G^*_{T,t}(X)_{U_a}|_{U_a \cap U_b} \cong  \G^*_{T,t}(X)_{U_b}|_{U_a \cap U_b}
\]
for each such intersection, such that the cocycle condition $\phi_{c,b} \circ \phi_{b,a} = \phi_{c,a}$ is satisfied whenever $U_a \cap U_b \cap U_c \neq \emptyset$.
\end{remark}

Note that we have either $X^b \subset X^a$ or $X^a \subset X^b$ whenever $U_{a} \cap U_{b} \neq \emptyset$, by the second property of an adapted cover.

\begin{theorem}
Let $X$ be a finite $T$-CW complex, and let $\U$ be a cover adapted to $\S(X)$. Let $a \leq b$ be points in $E_{T,t}$ and assume $U \subset U_{a} \cap U_{b}$ is an open subset. By the second property of an adapted cover, we may assume that $X^b \subset X^a$, with inclusion map $i_{b,a}$. The map 
\[
i_{b,a}^*\otimes \id: H^*_T(X^a) \otimes_{H_T} \O_{E_{T,t}}(U-a) \to H^*_T(X^b) \otimes_{H_T} \O_{E_{T,t}}(U-a)
\]
induced by restriction along $i_{b,a}$ is an isomorphism of $\O_{E_{T,t}}(U)$-modules.
\end{theorem}

\begin{proof}
See the proof of Theorem 2.3.3 in \cite{Chen}.
\end{proof}

\begin{remark}
Let $T(a,b) = \langle T(a),T(b) \rangle$ and let $U \subset U_a \cap U_b$ be an open subset. There is a natural isomorphism of $\Z/2\Z$-graded $\O(U)$-algebras given on $U \subset U_a$ by the composite
\begin{equation}\label{perry}
\begin{array}{rcl}
H^*_T(X^a) \otimes_{H_T} \O_{E_{T,t}}(U-a)& \xrightarrow{i_{b,a}^*\otimes \id}& H^*_T(X^b) \otimes_{H_T} \O_{E_{T,t}}(U-a) \\
&\longrightarrow& H^*_{T/T(a,b)}(X^b) \otimes_{H_{T/T(a,b)}} \O_{E_{T,t}}(U-a) \\
&\xrightarrow{\id \otimes \mathrm{tr}_{b-a}^*}& H^*_{T/T(a,b)}(X^b) \otimes_{H_{T/T(a,b)}} \O_{E_{T,t}}(U-b) \\
&\longrightarrow& H^*_T(X^b) \otimes_{H_T} \O_{E_{T,t}}(U-b),
\end{array}
\end{equation}
where the second and final maps are the change of group maps of Proposition \ref{changeh}. Denote the composite map by $\phi_{b,a}$.
\end{remark}

\begin{remark}
Our construction of the gluing maps $\phi_{b,a}$ differs from the construction given in Chen's thesis \cite{Chen}. Namely, in Proposition 2.3.4 and Definition 2.3.5 of that thesis, the gluing maps $\phi_{b,a}$ are defined much as in \eqref{perry}, only the third map is written as
\begin{equation}\label{sg}
H^*_{T/T(b)}(X^b) \otimes_{H_{T/T(b)}} \O_{E_{T,t}}(U-a) \xrightarrow{\id \otimes \mathrm{tr}_{b-a}^*} H^*_{T/T(b)}(X^b) \otimes_{H_{T/T(b)}} \O_{E_{T,t}}(U-b).
\end{equation}
However, such a map is not well defined, because $\mathrm{tr}_{b-a}^*$ does not in general preserve the $H_{T/T(b)}$-algebra structure, because $b-a$ is not always contained in $E_{T(b),t}$. To see this, suppose that $X$ is equal to a point, so that $\S(X) = \{T\}$, and let $a = [t_1/2]$ and $b = [0]$. Then $T(a) = \Z/2\Z$ and $T(b) = 1$ and $b - a$ is equal to $[-t_1/2]$, which is not in $E_{T(b),t} = 0$. \par

It is for this reason that we construct the gluing maps $\phi_{b,a}$ as in \eqref{perry}, using the change of groups map associated to $T \rightarrow T/T(a,b)$, instead of $T \rightarrow T/T(b)$. By definition of $T(a)$ and $T(b)$, we have that $b-a \in E_{T(a,b),t}$, and it follows that $\mathrm{tr}_{b-a}^*: \O(U-a) \to \O(U-b)$ is a map of $H_{T/T(a,b)}$-algebras, which fixes the problem. It is also important to note that $T(a,b)$ acts trivially on $X^b$, since $X^b \subset X^a$.
\end{remark}

The following result may be proved in the same way that Proposition 2.3.7 is proved in \cite{Chen}.

\begin{proposition}
The collection of maps $\{\phi_{b,a}\}$ satisfies the cocycle condition
\[
\phi_{c,b} \circ \phi_{b,a} = \phi_{c,a}
\]
whenever $U_a \cap U_b \cap U_c \neq \emptyset$.
\end{proposition}

\begin{definition}
We denote by $\G^*_{T,t}(X)$ the sheaf of $\Z/2\Z$-graded $\O_{E_{T,t}}$-algebras which is obtained by gluing together the sheaves $\G^*_{T,t}(X)_{U_a}$ via the maps $\phi_{b,a}$. 
\end{definition}

\begin{remark}
Up to isomorphism, the sheaf $\G^*_{T,t}(X)$ does not depend on the choice of $\U$ since any refinement of $\U$ is also adapted to $\S(X)$. More explicitly, given two covers $\U$ and $\U'$ adapted to $\S(X)$, one may take the common refinement $\U''$ and consider the theory defined using $\U''$. The resulting theory is then naturally isomorphic to those theories corresponding to $\U$ and $\U'$, since the maps induced by the refinement are isomorphisms on stalks.
\end{remark}

\begin{definition}
Define the reduced theory $\widetilde{\G}^*_{T,t}$ on a pointed, finite $T$-CW complex $X$ to be the kernel
\[
\ker(\G^*_{T,t}(X) \to \G^*_{T,t}(\pt))
\]
of the map induced by the inclusion of the basepoint $\pt \hookrightarrow X$.
\end{definition}

The following is adapted from Theorem 2.3.8 in \cite{Chen}. We reproduce the proof here as it is important for our main results.

\begin{proposition}\label{kerry}
Let $X$ be a pointed, finite $T$-CW complex. The assignment $X \mapsto \widetilde{\G}^*_{T,t}(X)$ defines a reduced $T$-equivariant elliptic cohomology theory in the sense of Definition \ref{tear}.
\end{proposition}

\begin{proof}
The sheaf $\widetilde{\G}^*_{T,t}(X)$ is coherent because $X$ is a finite $T$-CW complex, and $\G^*_{T,t}(X)$ may be computed locally using cellular cohomology. It is also a $\Z/2\Z$-graded $\O_{E_{T,t}}$-algebra, by construction. We show that the construction of $\widetilde{\G}^*_{T,t}(X)$ is functorial in $X$. Let $f: X \to Y$ be a map of pointed finite $T$-CW complexes and let $\U$ be a cover of $E_{T,t}$ which is adapted to $\S(f)$. For $a \in E_{T,t}$, the map $f$ induces a map $f_a: X^a \to Y^a$ by restriction. This induces a map 
\[
f_a^* \otimes \id: H^*_T(Y^a) \otimes_{H_T} \O(U - a) \to H^*_T(X^a) \otimes_{H_T} \O(U-a)
\]
for each $U \subset U_a$, which clearly commutes with the restriction maps of the sheaf. It is evident that the collection of such maps for all $a \in E_{T,t}$ glue well, and that identity maps and composition of maps are preserved, by the functoriality of Borel-equivariant cohomology and naturality of the isomorphism of Proposition \ref{changeh}. Thus, a map $f: X \to Y$ of pointed complexes gives rise to a unique map $f^*: \G^*_{T,t}(Y) \to \G^*_{T,t}(X)$, which induces a map on the kernels defining the reduced theory. It is straightforward to show that composition of maps and identity maps are preserved. 

Define a suspension isomorphism $\sigma: \widetilde{\G}^{*+1}_{T,t}(S^1 \wedge X) \to \widetilde{\G}^{*}_{T,t}(X)$ by gluing the maps 
\[
\sigma_a \otimes \id: \widetilde{H}^{*+1}_T(S^1 \wedge X^a) \otimes_{H^*_T} \O(U - a) \to \widetilde{H}^{*}_T(X^a) \otimes_{H^*_T} \O(U-a),
\]
where $\sigma_a$ is the suspension isomorphism of Borel-equivariant cohomology. The maps $\sigma_a \otimes \id$ glue well since $\sigma_a$ is natural, from which it also follows that $\sigma$ is natural. 

The properties of exactness and additivity may be checked on stalks
\[
\widetilde{\G}^*_{T,t}(X)_a = \widetilde{H}^*_T(X^a) \otimes_{H_T} \O_{E_{T,t},0}.
\]
This is clear, since Borel-equivariant cohomology satisfies these properties, and tensoring with $\O_{E_{T,t},0} \cong \O_{\t_\C,0}$ is exact. Finally, homotopy invariance follows from the homotopy invariance of Borel equivariant cohomology.
\end{proof}

\section{The $\C^\times \times \SL_2(\Z)$-equivariant complex manifold $E_T$}\label{water}

In this section, we work out the details of the picture sketched by Rezk in Section 2.12 of \cite{Rezk} (see also Etingof and Frenkel \cite{EF}). Namely, we construct a $\C^\times \times \SL_2(\Z)$-equivariant complex manifold $E_T$ as an equivariant fiber bundle over $\X^+$, such that the fiber over $t$ is equal to $E_{T,t} = \check{T} \otimes E_t$. The manifold $E_T$ will be the base space of the $\C^\times \times \SL_2(\Z)$-equivariant sheaf $\E_T(X)$ that we construct in Section \ref{brew}.

\begin{remark}
Consider the semidirect product group 
\[
\SL_2(\Z) \ltimes \bT^2
\]
where $\SL_2(\Z)$ acts on $\bT^2 = \R^2/\Z^2$ by left multiplication. The group operation is given by 
\[
(A',t')(A,t) = (A'A, A^{-1}t' + t),
\]
so that the inverse of $(A,t)$ is $(A^{-1},-At)$. We may think of $\SL_2(\Z) \ltimes \bT^2$ as the group of orientation-preserving diffeomorphisms 
\[
\begin{array}{rcl}
(A,t): \bT^2 &\longrightarrow & \bT^2 \\
s &\longmapsto & As + t.
\end{array}
\]
Let $L^2T$ be the topological group of smooth maps $\bT^2 \to T$, with group multiplication defined pointwise. A diffeomorphism $(A,s)$ acts on a loop $\gamma \in L^2T$ from the left by
\[
(A,t)\cdot \gamma(s) = \gamma(A^{-1}s - At).
\]
\end{remark}

\begin{definition}
Following \cite{Rezk}, define the \textit{extended double loop group of $T$} as the semidirect product 
\[
\widetilde{L^2T} := (\SL_2(\Z) \ltimes \bT^2) \ltimes L^2T
\]
with group operation 
\begin{equation}\label{cute}
(A',t',\gamma'(s))(A,t,\gamma(s)) = (A'A, A^{-1}t' + t, \gamma'(As+ t)+ \gamma(s)).
\end{equation}
One may think of an element $(A,t,\gamma) \in \widetilde{L^2T}$ as the automorphism 
\[
\begin{tikzcd}
\bT^2 \times T \ar[r,"{\phi}"] \ar[d,two heads] & \bT^2 \times T \ar[d,two heads] \\
\bT^2 \ar[r,"{(A,t)}"] & \bT^2
\end{tikzcd}
\]
covering the diffeomorphism $(A,t)$ of $\bT^2$, where $\phi(r,s)$ is equal to $\gamma(r) + s$. It is easily verified that the inverse of $(A,t,\gamma(s))$ is equal to $(A^{-1},-At,- \gamma(A^{-1}s-At))$.
\end{definition}

\begin{remark}\label{wei1}
For a finite $T$-CW complex $X$, the extended double loop group $\widetilde{L^2T}$ acts on the double loop space 
\[
L^2X := \Map(\bT^2,X) 
\] 
by
\begin{equation}
(A,t,\gamma(s))\cdot f(s) = \gamma(A^{-1}s- At)\cdot f(A^{-1}s-At)
\end{equation}
for all $(A,t,\gamma) \in \widetilde{L^2T}$ and $f \in L^2X$. One may also think of this action in the following way. Consider the mapping space $\Map_T(\bT^2 \times T,X)$ of $T$-equivariant maps from the trivial $T$-bundle over $\bT^2$ into $X$. There is a homeomorphism
\[
\Map_T(\bT^2 \times T,X) \xrightarrow{\sim} \Map(\bT^2,X) = L^2X
\]
given by restriction to the subspace $\bT^2 \times \{0\} \subset \bT^2 \times T$. The induced action of $(A,t,\gamma) \in \widetilde{L^2T}$ on the equivariant mapping space is given by pulling back along the commutative diagram
\[
\begin{tikzcd}
\bT^2 \times T \ar[r,"{\phi}"] \ar[d,two heads] & \bT^2 \times T \ar[d,two heads] \ar[r] & X\\
\bT^2 \ar[r,"{(A,t)}"] & \bT^2,
\end{tikzcd}
\]
where $\phi(r,s) = \gamma(r) + s$. Thus we see that the action groupoid $\widetilde{L^2T} \backslash\!\!\backslash L^2X$ is isomorphic to the subgroupoid of $\Map(\bT^2,T \backslash\!\!\backslash X)$ wherein the $T$-bundle over $\bT^2$ is trivial.
\end{remark}

\begin{definition}\label{wei}
Consider the subgroup 
\[
\bT^2 \times T \subset \bT^2 \ltimes L^2T
\]
where the translations $\bT^2$ act trivially on the subgroup of constant loops $T \subset L^2T$. One sees that this is a maximal torus in $\widetilde{L^2T}$ by noting that the intersection of $\SL_2(\Z)$ with the identity component of $\widetilde{L^2T}$ is trivial, and that a nonconstant loop in $L^2T$ does not commute with $\bT^2$. Let $N_{\widetilde{L^2T}}(\bT^2 \times T)$ be the normaliser of $\bT^2 \times T$ in $\widetilde{L^2T}$. The Weyl group associated to $\bT^2 \times T \subset \widetilde{L^2T}$ is defined as
\[
W_{\widetilde{L^2T}} = W_{\widetilde{L^2T}}(\bT^2 \times T) := N_{\widetilde{L^2T}}(\bT^2 \times T)/(\bT^2 \times T).
\]
\end{definition}

\begin{remark}
In the following proposition, we will consider the subgroup
\[
\SL_2(\Z) \ltimes \check{T}^2 \subset \widetilde{L^2T}
\]
where $m \in \check{T}^2$ is identified with the loop $\gamma(s) = ms \in L^2T$ via
\[
\check{T}^2 := \Hom(\bT,T)^2 \cong \Hom(\bT^2,T) \subset L^2T.
\]
The group operation, induced by that of $\widetilde{L^2T}$, is given by
\[
(A',m')(A,m) = (A'A, m'A +m) \in \SL_2(\Z) \ltimes \check{T}^2,
 \]
and the inverse of $(A,m)$ is given by $(A^{-1},-mA^{-1})$. 
\end{remark}

\begin{proposition}
The subgroup $\SL_2(\Z) \ltimes \check{T}^2 \subset \widetilde{L^2T}$ is contained in $N_{\widetilde{L^2T}}(\bT^2 \times T)$, and the composite map
\[
\SL_2(\Z) \ltimes \check{T}^2 \hookrightarrow N_{\widetilde{L^2T}}(\bT^2\times T) \twoheadrightarrow W_{\widetilde{L^2T}}(\bT^2 \times T)
\]
is an isomorphism.
\end{proposition}

\begin{proof}
Let $(A,m) \in \SL_2(\Z) \ltimes \check{T}^2$. A straightforward calculation using \eqref{cute} shows that 
\[
(A,0,m) (1,r,t) (A^{-1},0,-mA^{-1}) = (1,Ar,t + mr) \in \bT^2 \times T,
\]
which proves the first assertion. For the second assertion, it suffices to define an inverse to the composite map of the proposition. Let $g$ be an arbitrary element in $N_{\widetilde{L^2T}}(\bT^2\times T)$ and let $[g]$ be its image in $W_{\widetilde{L^2T}}$. By definition of $W_{\widetilde{L^2T}}$, we may translate $g$ by elements of $\bT^2$, and act on $g$ by constant loops, without changing $[g]$. Therefore, there exists $\gamma \in L^2T$ with $\gamma(0,0) = 1$ such that 
\[
[g] = [(A,0,\gamma)] \in W_{\widetilde{L^2T}},
\]
for some $A \in \SL_2(\Z)$. We will now show that $\gamma \in \check{T}^2$, and finally that $[g] \mapsto (A,\gamma)$ is a well defined inverse to the composite map. Using \eqref{cute} again, for any $(r,t) \in \bT^2 \times T$, we have
\begin{equation}\label{afff}
(A,0,\gamma(s)) (r,t)(A,0,- \gamma(s)) = (1,Ar,\gamma(r+A^{-1}s)+ t - \gamma(A^{-1}s))  \in \bT^2 \times T.
\end{equation}
It follows that $\gamma(r+A^{-1}s) - \gamma(A^{-1}s)$ does not depend on $s$. Thus,
\[
\gamma(r+A^{-1}s)- \gamma(A^{-1}s) = \gamma(r)
\]
for all $r,s \in \bT^2$, and setting $s = As'$ shows that $\gamma(r)+ \gamma(s') = \gamma(r+s')$ for all $r,s' \in \bT^2$. Therefore, $\gamma$ is a group homomorphism, which means that it lies in $\check{T}^2$. The map $[g] \mapsto (A,\gamma)$ is well defined, since $g \in \bT^2 \times T$ allows us to choose $A = 1$ and $\gamma = 1$, and is evidently a group homomorphism which is inverse to the composite map of the proposition. This completes the proof.
\end{proof}

\begin{remark}\label{ulmo}
It follows directly from equation \eqref{afff} that the action of $W_{\widetilde{L^2T}}$ on $\bT^2 \times T$ is given by
\[
(A,m) \cdot (r,t) = (Ar,t+mr).
\]
The induced action of $W_{\widetilde{L^2T}}$ on the complexified Lie algebra $\C^2 \times \t_\C$ is given by the same formula, in which case we write it as
\[
(A,m) \cdot (t,x) = (At,x+mt).
\]
We consider the $\C^\times$-action on $\C^2\times \t_\C$ given by
\[
\lambda \cdot (t,x) = (\lambda^2 t, \lambda^2 x),
\]
which will be important to the construction of the equivariant sheaf in Section \ref{brew} (compare also the map $\alpha$ in Remark \ref{ainur}). Note that the $\C^\times$-action on $\C^2\times \t_\C$ commutes with the $W_{\widetilde{L^2T}}$-action.
\end{remark}

\begin{remark}\label{melkor}
Recall from Section \ref{bread} the compex analytic space 
\[
\X^+ := \{(t_1,t_2) \in \C^2 \, |\, \mathrm{Im}(t_1/t_2) > 0, t_2 \neq 0\}.
\]
It is straightforward to verify that the action of $\C^\times \times W_{\widetilde{L^2T}}$ on $\C^2 \times \t_\C$ given in the previous remark preserves $\X^+ \times \t_\C$. Furthermore, the action of $\check{T}^2$ is free and properly discontinuous. The quotient map
\[
\zeta_T: \X^+ \times \t_\C \twoheadrightarrow \check{T}^2 \backslash (\X^+ \times \t_\C),
\]
is therefore a covering map of complex manifolds.
\end{remark}

\begin{remark}
Consider again the complex manifold
\[
E_T := \check{T}^2 \backslash (\X^+ \times \t_\C)
\]
from the previous remark. Since $\check{T}^2$ is a normal subgroup of $\C^\times \times W_{\widetilde{L^2T}}$, there is a residual $\C^\times \times \SL_2(\Z)$-action which descends to $E_T$. We denote by $\M_T$ the quotient stack
\[
\M_T := [(\C^\times \times \SL_2(\Z)) \backslash\!\!\backslash E_T].
\]
Recall the moduli stack $\M$ of elliptic curves from Remark \ref{dominic}. We view $\M_T$ as being equipped with the map $\M_T \to \M$ which is induced by the $\C^\times \times \SL_2(\Z)$-equivariant projection $E_T \twoheadrightarrow \X^+$. The fiber of $E_T$ over $t \in \X^+$ is given by 
\[
(\check{T}t_1 + \check{T}t_2)\backslash \t_\C = \check{T} \otimes (\Lambda_t\backslash \C) = \check{T}^2 \otimes E_t = E_{T,t},
\]
and the restriction of $\zeta_T$ to the fiber over $t \in \X^+$ is the map
\[
\zeta_{T,t}: \t_\C \twoheadrightarrow E_{T,t}
\]
of Remark \ref{fad}. 
\end{remark}

\begin{remark}\label{grut}
Recall that $t \in \X^+$ determines an identification $\R^2 \cong \R t_1 + \R t_2 = \C$, which we denote by $\xi_t$. The maps $\xi_t$ assemble over all $t \in \X^+$ to give an isomorphism
\[
\xi: \X^+ \times \R^2 \longrightarrow \X^+ \times \C
\]
of real vector bundles over $\X^+$. Recall that $E_t = \Lambda_t \backslash \C$, by definition. The free $\Z^2$-action on $\R^2$ corresponds via $\xi$ to the natural action of $\Lambda_t$ on $\C$, so that $\xi$ induces a commutative diagram
\begin{equation}\label{seer}
\begin{tikzcd}
\X^+ \times \R^2 \ar[r,"\sim","\xi"'] \ar[d,two heads,"\id \times \exp_{\bT^2}"]& \X^+ \times \C \ar[d, two heads, "{\zeta}" ]\\
\X^+ \times \bT^2 \ar[r,"\sim","\chi"'] & E
\end{tikzcd}
\end{equation} 
of real manifolds, where $\chi$ is the map induced by $\xi$ on $\Z^2$-orbits. Recall the definitions of Remark \ref{fad} and let $K$ be an arbitrary compact abelian group. We apply the functor $\Hom(\hat{K},-)$ in a fiberwise manner to diagram \eqref{seer}, to yield the commutative diagram
\begin{equation}\label{deer}
\begin{tikzcd}
\X^+ \times \Lie(K) \times \Lie(K) \ar[r,"\sim","\xi_K"'] \ar[d,two heads,"\id \times \exp_{K\times K}"]& \X^+ \times \Lie(K)_\C \ar[d, two heads, "{\zeta_{K}}" ]\\
\X^+ \times K \times K \ar[r,"\sim","\chi_K"'] & E_{K}.
\end{tikzcd}
\end{equation} 
If $K$ is a torus $T$, then we may achieve the same result by instead applying $\check{T} \otimes - $ fiberwise to diagram \eqref{seer}. Diagram \eqref{deer} is clearly natural in $K$.
\end{remark}


\section{An open cover of $E_T$ adapted to $X$}\label{soup}

In this section, we prove several technical results which will mainly be used to obtain a local description of the sheaf $\H^*_{\bT^2 \times T}(L^2X)$ in Section \ref{low}. We begin by defining, for a finite $T$-CW complex $X$, an open cover of the compact Lie group $T \times T$ adapted to $X$. We show that such a cover exists, and that it induces an open cover of the total space $E_T$ via the isomorphism $\chi_T: \X^+ \times T \times T \cong E_T$. We also show that the restriction of the cover to $E_{T,t}$ is adapted to $X$ in the sense of Definition \ref{ross}. Finally, we establish some properties of the open cover which will be useful in later sections, and we give an example of a cover adapted to the representation sphere $S_\lambda$. 

\begin{definition}
If $\S$ is a finite set of closed subgroups of $T$, we can define a relation on the points of $T\times T$ by saying that $(a_1,a_2) \leq_\S (b_1,b_2)$ if $(b_1,b_2) \in K \times K$ implies $(a_1,a_2) \in K\times K$, for any $K \in \S$. This relation is obviously reflexive and transitive, but not symmetric.
\end{definition}

\begin{definition}\label{rachel}
We use the notation of Definition \ref{ruff}, and say that an open set $U$ in $T \times T$ is \textit{small} if $\exp_{T\times T}^{-1}(U)$ is a disjoint union of connected components $V$ such that $\exp_{T\times T}|_V: V \to U$ is a bijection for each $V$. An open cover $\U = \{U_{a_1,a_2}\}$ of $T \times T$ indexed by the points of $T\times T$ is said to be \textit{adapted to $\S$} if it has the following properties:
\begin{enumerate}
\item $(a_1,a_2) \in U_{a_1,a_2}$, and $U_{a_1,a_2}$ is small.
\item If $U_{a_1,a_2} \cap U_{b_1,b_2} \neq \emptyset$, then either $(a_1,a_2) \leq_\S (b_1,b_2)$ or $(b_1,b_2) \leq_\S (a_1,a_2)$.
\item If $(a_1,a_2) \leq_\S (b_1,b_2)$, and for some $K \in \S$, we have $(a_1,a_2) \in K \times K$ but $(b_1,b_2) \notin K\times K$, then $U_{b_1,b_2} \cap K \times K = \emptyset$.
\item Let $(a_1,a_2)$ and $(b_1,b_2)$ lie in $K \times K$ for some $K \in \S$. If $U_{a_1,a_2} \cap U_{b_1,b_2} \neq \emptyset$, then $(a_1,a_2)$ and $(b_1,b_2)$ belong to the same connected component of $K\times K$.
\end{enumerate}
If $\U$ is adapted to $\S = \S(X)$ then we say that $\U$ is adapted to $X$. If $\S = \S(f)$ and $\U$ is adapted to $\S$ then we say that $\U$ is adapted to $f$. If $\S$ is understood, then we just write $\leq$ for $\leq_\S$. 
\end{definition}

Our proof of the following result is based on the proof of Proposition 2.5 in \cite{Rosu03}.

\begin{lemma}\label{construu}
For any finite set $\S$ of subgroups of $T$, there exists an open cover $\U$ of $E_T$ adapted to $\S$. Any refinement of $\U$ is also adapted to $\S$.
\end{lemma}

\begin{proof}
Since a compact abelian group has finitely many components, the set 
\[
\S^0 = \{D \subset T \times T \, | \, \text{$D$ is a component of $K \times K$ for some $K \in \S$}\}
\]
is finite. We choose a metric on $T \times T$ and denote it by $d$. 
If $(a_1,a_2) \in D$ for all $D \in \S^0$, then define $U_{a_1,a_2}$ to be an open ball centered at $(a_1,a_2)$ with radius $r = \frac{1}{2}$. Otherwise, define $U_{a_1,a_2}$ to be an open ball centered at $(a_1,a_2)$, with radius  
\[
r = \frac{1}{2} \min \{d((a_1,a_2),D) \, | \, (a_1,a_2) \notin D \},
\]
where 
\[
d((a_1,a_2),D) = \min \{d((a_1,a_2),(b_1,b_2))\, | \, (b_1,b_2)\in D\}.
\]
The open cover $\U$ of $T \times T$ thus constructed clearly satisfies the first condition of an adapted cover.  

Furthermore, if there exist distinct components $D,D' \in \S^0$ such that $(a_1,a_2)$ is in $D$ but not in $D'$, and $(b_1,b_2)$ is in $D'$ but not in $D$, then $U_{a_1,a_2} \cap U_{b_1,b_2}$ is empty by construction. If $D$ and $D'$ correspond to distinct elements of $\S$, then $(a_1,a_2)$ and $(b_1,b_2)$ do not relate under the ordering, and the previous statement implies the contrapositive of the second condition. If $D$ and $D'$ correspond to the same element of $\S$, then it implies the contrapositive of the fourth condition. 

The third condition holds because $(b_1,b_2) \notin K\times K$ implies that $U_{b_1,b_2}$ does not intersect $K\times K$, by construction. It is clear that any refinement of $\U$ will also satisfy all four conditions.
\end{proof}

\begin{lemma}\label{hard2}
Let $\U$ be an open cover of $T\times T$ adapted to $\S$. If $(b_1,b_2) \in U_{a_1,a_2}$, then $(a_1,a_2) \leq (b_1,b_2)$.
\end{lemma}

\begin{proof}
Suppose that $(b_1,b_2) \in U_{a_1,a_2}$ and $(a_1,a_2) \leq (b_1,b_2)$ does not hold. This implies two things. Firstly, by the second property of an adapted cover, we have $(b_1,b_2) \leq (a_1,a_2)$. Secondly, by definition of the relation, there must exist some $K \in \S$ such that $(b_1,b_2) \in K \times K$ and $(a_1,a_2) \notin K \times K$. Together, the two statements imply that $U_{a_1,a_2} \cap K \times K = \emptyset$, by the third property of an adapted cover. This contradicts the assumption that $(b_1,b_2) \in U_{a_1,a_2}$, since $(b_1,b_2) \in K \times K$.
\end{proof}

\begin{definition}\label{siss}
Let $\U = \{U_{a_1,a_2}\}$ be an open cover adapted to $\S$. We denote by $V_{x_1,x_2} \subset \t \times \t$ the open subset which is the component of $\exp_{T\times T}^{-1}(U_{\exp_{T\times T}(x_1,x_2)})$ containing $(x_1,x_2)$.
\end{definition}

The following lemma is a strengthening of the fourth property of an adapted cover, which we will need in the next section. 

\begin{lemma}\label{hard3}
Let $\U$ be a cover of $T\times T$ adapted to $\S$. Let $(a_1,a_2),(b_1,b_2) \in T\times T$ with open neighbourhoods $U_{a_1,a_2},U_{b_1,b_2} \in \U$, and let $(x_1,x_2),(y_1,y_2) \in \t \times \t$ such that $\exp_T(x_i) = a_i$ and $\exp_T(y_i) = b_i$. Let $K \in \S$ and suppose $(a_1,a_2),(b_1,b_2) \in K\times K$. If
\[
V_{x_1,x_2} \cap V_{y_1,y_2} \neq \emptyset, 
\]
then $(x_1,x_2)$ and $(y_1,y_2)$ lie in the same component of $\exp_{T\times T}^{-1}(K\times K)$.
\end{lemma}

\begin{proof}
Since $V_{x_1,x_2} \cap V_{y_1,y_2} \neq \emptyset$, we have $U_{a_1,a_2} \cap U_{b_1,b_2} \neq \emptyset$. Therefore, $(a_1,a_2)$ and $(b_1,b_2)$ lie in the same component $D$ of $K\times K$, by the fourth property of an adapted cover, so that $(x_1,x_2),(y_1,y_2) \in \exp_{T\times T}^{-1}(D)$. We have
\begin{equation}
\begin{array}{rcl}
\exp_{T\times T}^{-1}(D) &=& \{(x_1,x_2)\} + \check{T} \times \check{T} + \mathrm{Lie}(K) \times \mathrm{Lie}(K)\\
&=& \{(x_1,x_2)\} + \check{T}/\check{K} \times \check{T}/\check{K} + \mathrm{Lie}(K) \times \mathrm{Lie}(K). \\
\end{array}
\end{equation}
Note that, if $A$ and $B$ are subsets of a topological group, then we use the notation $A + B$ to mean the set $\{a + b \, | \, a \in A, b\in B\}$. Now suppose that $(x_1,x_2)$ and $(y_1,y_2)$ lie in different components of $\exp_{T\times T}^{-1}(D)$. Then $\check{T}/\check{K}$ is nontrivial, and we have
\[
(x_1,x_2) = (x_1 + m_1 + h_1, x_2 + m_2 + h_2) \quad \mathrm{and} \quad (y_1,y_2) = (x_1 + m_1' + h_1', x_2 +m_2' + h_2')
\]
for some $h_1,h_2,h_1',h_2' \in \Lie(K)$ and distinct $(m_1,m_2),(m_1',m_2') \in \check{T}/\check{K}\times \check{T}/\check{K}$.

Choose a metric $d$ on $\t \times \t$. We have
\[
\begin{array}{rcl}
d((x_1,x_2),(y_1,y_2)) &=& |(x_1,x_2)-(y_1,y_2)|  \\
&=& |(m_1-m_1' + h_1 - h_1',m_2-m_2' + h_2 - h_2')|\\
&\geq &|m_1-m_1',m_2-m_2' | \\
\end{array}
\]
where the inequality holds since $\check{T}/\check{K}$ is orthogonal to Lie$(K)$. Since $(m_1,m_2) \neq (m_1',m_2')$, we have
\[
d((x_1,x_2),(y_1,y_2)) \geq 1.
\]

By the first property of an adapted cover, $U_{a_1,a_2}$ is small, which means that $V_{x_1,x_2}$ is contained in the interior of a ball at $(x_1,x_2)$ with radius $\frac{1}{2}$. The same is true for $V_{y_1,y_2}$. But this means that $V_{x_1,x_2}$ and $V_{y_1,y_2}$ cannot intersect, since $d((x_1,x_2),(y_1,y_2)) \geq 1$, so we have a contradiction. Therefore, $(x_1,x_2)$ and $(y_1,y_2)$ must lie in the same connected component of $\exp_{T\times T}^{-1}(D)$, and hence the same connected component of $\exp_{T\times T}^{-1}(K\times K)$.
\end{proof}

\begin{definition}\label{defo}
Let $\U$ be an open cover of $T \times T$ adapted to $\S$. For each $(t,x)  \in \X^+ \times \t_\C$, writing $x = x_1t_1 + x_2t_2$, we define an open subset
\[
V_{t,x} := \xi_T(\X^+ \times V_{x_1,x_2}) \subset \X^+ \times \t_\C,
\]
so that $(t,x) \in V_{t,x}$. Note that $V_{t,x} = V_{t',x'}$ whenever we have $x = x_1t_1 +x_2t_2$ and $x' = x_1t_1'+x_2t_2'$. The set 
\[
\{V_{t,x}\}_{(t,x) \in \X^+ \times \t_\C}
\]
is an open cover of $\X^+ \times \t_\C$ (with some redundant elements). The set
\[
\{\zeta_T(V_{t,x})\}_{(t,x) \in \X^+ \times \t_\C}
\]
is an open cover of $E_T$.
\end{definition}

\begin{definition}\label{sag}
Let $\U$ be an open cover of $T \times T$ adapted to $\S$. Given $a \in E_{T,t}$, define the open subset
\[
U_a := \zeta_{T}(V_{t,x}) \cap E_{T,t}
\]
where $x \in \t_\C$ is any element such that $a= \zeta_{T,t}(x)$, so that $a \in U_a$. The set $\{U_a\}_{a\in E_{T,t}}$ is an open cover of $E_{T,t}$.
\end{definition}

\begin{lemma}
Let $\U$ be an open cover of $T \times T$ which is adapted to $\S$. The open cover of $E_{T,t}$ in Definition \ref{sag} is adapted to $\S$ in the sense of Definition \ref{ross}.
\end{lemma}

\begin{proof}
We have
\[
\begin{array}{rcl}
U_a &:=& \zeta_T (V_{t,x}) \cap E_{T,t}  \\
&=& (\zeta_T \circ \xi_T)(\X^+ \times V_{x_1,x_2}) \cap E_{T,t}  \\
&=& (\chi_T \circ (\id_{\X^+} \times \exp_{T\times T}))(\X^+ \times V_{x_1,x_2}) \cap E_{T,t}  \\
&=& \chi_T(\X^+ \times U_{a_1,a_2}) \cap E_{T,t} \\
&=& \chi_{T,t}(U_{a_1,a_2})
\end{array}
\]
where $a_i = \exp_T(x_i)$. Therefore, the open cover $\{U_a\}_{a\in E_{T,t}}$ of $E_{T,t}$ corresponds exactly to the open cover $\U$ of $T \times T$ via the isomorphism
\[
\chi_{T,t}: T \times T \cong E_{T,t}
\]
of real Lie groups. By naturality of diagram \eqref{deer} it is clear that the properties of an adapted cover in the sense of Definition \ref{ross} are equivalent to the properties in Definition \ref{rachel}. Since $\U$ is adapted to $\S$ in the sense of Definition \ref{rachel}, the result now follows.
\end{proof}

\begin{example}\label{coverex}
Let $T = \R/\Z$, and for $\lambda \in \hat{T}$ set $X$ equal to the representation sphere $S_\lambda$ associated to $\lambda$ (see Example \ref{repsphere}). Recall that $X$ has a $T$-CW complex structure with a $0$-cell $T/T \times \D_N^0$ at the north pole, a $0$-cell $T/T \times \D_S^0$ at the south pole, and a 1-cell $T \times \D^1$ of free orbits connecting the two poles. Thus,
\[
\S(X) = \{T, 1\}.
\]
The relations $\leq_{\S(X)}$ between the points of $T \times T$ are:
\begin{enumerate}
\item $(0,0) \leq (a_1,a_2)$ for all $(a_1,a_2) \in T \times T$; and
\item $(a_1,a_2) \leq (b_1,b_2)$ for all $(a_1,a_2),(b_1,b_2) \in T \times T - \{(0,0)\}$.
\end{enumerate}
As in the proof of Lemma \ref{construu}, we can easily construct an open cover of $T \times T$ which is adapted to $\S(X)$. Note that $\S^0 = \S$ in this case. Let $\U$ denote the open cover consisting of open balls $U_{a_1,a_2}$ centered at $(a_1,a_2)$ with radius 
\[
r_{a_1,a_2} = \begin{cases} \frac{1}{2} \sqrt{a_1^2 + a_2^2}  &\mbox{if } a_1 \neq 0 \mbox{ or } a_2 \neq 0 \\ 
\frac{1}{2} & \mbox{if } (a_1,a_2) = (0,0)\end{cases},
\]
where we have identified $a_1,a_2$ with their unique representatives in $[ 0,1)$. It is easily verified that $\U$ is a cover adapted to $\S(X)$. Now, for $a = \chi_{T,t}(a_1,a_2)$ and $x = x_1t_1 +x_2t_2$ such that $\zeta_{T,t}(x) = a$, the open set
\[
V_{x_1,x_2} \subset \R^2
\]
is an open ball of radius $\frac{1}{2}$ if $(x_1,x_2) \in \Z^2$, and an open neighbourhood not intersecting $\Z^2$ otherwise. 
\end{example}

\section{The Borel-equivariant cohomology of double loop spaces}\label{low}

In Definition \ref{bochum}, we defined for a finite $T$-CW complex $X$ the holomorphic sheaf $\H^*_T(X)$ of $\Z/2\Z$-graded $\O_{\t_\C}$-algebras. We would like to apply this definition to produce the main ingredient in our construction of elliptic cohomology: the holomorphic sheaf $\H^*_{\bT^2 \times T}(L^2X)$, which depends on the $\bT^2 \times T$-equivariant space $L^2X$. However, since $L^2X$ is not compact, we cannot apply Definition \ref{bochum} directly. To solve this problem, we take inspiration from a construction that was given by Kitchloo in \cite{Kitch1} (see Definition 3.2). Namely, we tensor the cohomology ring of each finite subcomplex of $L^2X$ with the ring of holomorphic functions on $\X^+ \times \t_\C$, and then define $\H^*_{\bT^2 \times T}(L^2X)$ as the inverse limit over the maps induced by the inclusions of subcomplexes. While this definition is at first obscure, it will become clearer during remainder of this section when we obtain a nice description of the local values of $\H^*_{\bT^2 \times T}(L^2X)$.

\begin{remark}
By Theorem 1.1 in \cite{LMS}, the double loop space $L^2X = \Map(\bT^2,X)$ of $X$ is weakly $\bT^2 \times T$-homotopy equivalent to a $\bT^2 \times T$-CW complex $Z$.\footnote{To apply this theorem to our situation, set $G = \bT^2 \times T$. Now let $G$ act on $\bT^2$ via the projection to the first factor, and on $X$ via the projection to the second factor.} From now on, when we speak of a $\bT^2 \times T$-CW structure on $L^2X$, we will mean the replacement complex $Z$. In this situation, we will abuse notation and write $L^2X$ for $Z$.
\end{remark}

\begin{definition}\label{varlimm}
Recall the notation of Definition \ref{bochum}. We define the sheaf $\H^*_{\bT^2 \times T}(L^2X)$ of $\Z/2\Z$-graded $\O_{\X^+ \times \t_\C}$-algebras as
\[
\H^*_{\bT^2 \times T}(L^2X) := \varprojlim_{Y \subset L^2X} \H^*_{\bT^2 \times T}(Y)_{\X^+ \times \t_\C}
\]
where the inverse limit is in the category of $\Z/2\Z$-graded $\O_{\X^+ \times \t_\C}$-algebras, and runs over all finite $\bT^2 \times T$-CW subcomplexes $Y$ of $L^2X$.
\end{definition}

\begin{remark}
The inverse limit may be computed in the category of presheaves, so that the value of the inverse limit sheaf $\H^*_{\bT^2 \times T}(L^2X)$ on an open subset $U \subset \X^+ \times \t_\C$ is 
\[
\varprojlim_{Y \subset L^2X} H^*_{\bT^2 \times T}(Y) \otimes_{H_{\bT^2 \times T}} \O_{\X^+ \times \t_\C}(U).
\]
\end{remark}

\begin{remark}
If $X = \pt$ then $\H^*_{\bT^2 \times T}(L^2X) = \O_{\X^+ \times \t_\C}$, by construction.
\end{remark}

The sheaf $\H^*_{\bT^2 \times T}(L^2X)$ depends a priori on all finite subcomplexes $Y$ of $L^2X$, which are of arbitrarily large dimension. Since this is difficult to analyse, it is useful to apply Theorem \ref{localisationn} to describe the stalk $\H^*_{\bT^2 \times T}(Y)_{(t,x)}$ in terms of the cohomology of the subspace $Y^{t,x} \subset Y$ of loops fixed by a certain subgroup $T(t,x) \subset \bT^2 \times T$ (cf. Theorem 3.3 in \cite{Kitch1}). Doing this for each $Y$, we can then describe $\H^*_{\bT^2 \times T}(L^2X)_{(t,x)}$ as an inverse limit over a class of subspaces of $L^2X$ which is hopefully more tractable than before. It turns out that this is a fruitful line of attack, because the subgroup $T(t,x)$ is big enough to fix only a relatively small number of loops in $L^2X$. The main result of this section is that this description holds not only on stalks, but also on the open sets $\{V_{t,x}\}$ defined in the previous section. Before we arrive at the main result, we must prove five technical lemmas about the groups $T(t,x)$, the spaces $Y^{t,x}$ and the open sets $V_{t,x}$.

\begin{definition}
Let $(t,x)$ be a point in 
\[
\X^+ \times \t_\C\subset \Lie(\bT^2 \times T)_\C. 
\]
Recall that a closed subgroup $K \subset \bT^2 \times T$ induces an inclusion of complexified Lie algebras $\Lie(K)_\C \subset \Lie(\bT^2 \times T)_\C$. Define
\[
T(t,x) := \bigcap_{(t,x) \in \mathrm{Lie}(K)_\C} K
\]
as the intersection of closed subgroups $K \subset \bT^2 \times T $ whose complexified Lie algebra contains $(t,x)$. For a $\bT^2 \times T$-space $Y$, denote by $Y^{t,x}$ the subspace of points fixed by $T(t,x)$.
\end{definition}

\begin{remark}\label{push}
We give an explicit description of $T(t,x)$ for a given $(t,x) \in \X^+ \times \t_\C$. The identification
\[
\Lie(\bT^2 \times T)_\C =  \Lie(\bT^2 \times T) \otimes_\R \C = \Lie(\bT^2 \times T) t_1 + \Lie(\bT^2 \times T) t_2  
\]
induced by $\C = \R t_1 + \R t_2$ is clearly natural in subgroups of $\bT^2 \times T$. For a closed subgroup $K \subset \bT^2 \times T$, the complex Lie algebra $\Lie(K)_\C$ therefore contains 
\[
(t,x) = (t_1,t_2,x_1t_1+x_2t_2) = (1,0,x_1) t_1 + (0,1,x_2) t_2 \in \Lie(\bT^2 \times T)_\C,
\]
if and only if 
\[
(1,0,x_1),(0,1,x_2) \in \mathrm{Lie}(K).
\] 
Therefore, $T(t,x)$ is the smallest subgroup whose Lie algebra contains both $(1,0,x_1)$ and $(0,1,x_2)$, which is 
\[
T(t,x) = \langle \exp_{\bT^2 \times T}(r_1,0,x_1 r_1),\exp_{\bT^2 \times T}(0,r_2,x_2 r_2) \rangle_{r_1,r_2 \in \R}.
\]
Moreover, we also have 
\[
T(t,x) \cap T = \langle \exp_T(x_1),\exp_T(x_2) \rangle,
\]
where we have identified $T$ with its image under the inclusion $T \hookrightarrow \bT^2 \times T$. This holds since the intersection of $T(t,x)$ with $T \subset \bT^2 \times T$ consists exactly of those elements of $T(t,x)$ for which $r_1,r_2 \in \Z$.  
\end{remark}

\begin{remark}\label{pusha}
Recall the subgroup $T(a) \subset T$ of Definition \ref{teed}. Let $a = \zeta_T(t,x)$ and write $x = x_1t_1 + x_2t_2$. By diagram \eqref{deer}, it is clear that $E_{K,t}$ contains $a$ if and only if 
\[
\exp_T(x_1),\exp_T(x_2) \in K.
\]
Therefore, $T(a)$ is the closed subgroup 
\[
\langle \exp_T(x_1), \exp_T(x_2) \rangle \subset T.
\] 
By Remark \ref{push}, we therefore have that $T(a) = T \cap T(t,x)$.
\end{remark}

\begin{lemma}\label{surj}
Let $\zeta_T(t,x) = a$. There is a short exact sequence of compact abelian groups 
\[
0 \to T(a) \to T(t,x) \to \bT^2 \to 0
\]
where $T(t,x) \to \bT^2$ is the map induced by the projection $\bT^2 \times T \twoheadrightarrow \bT^2$.
\end{lemma}

\begin{proof}
The projection of $T(t,x) \subset \bT^2 \times T$ onto $\bT^2$ is surjective by the description in Remark \ref{push}, and has kernel $T(t,x) \cap T = T(a)$ by Remark \ref{pusha}.
\end{proof}

\begin{notation}\label{nerf}
Write
\[
p_a: T \twoheadrightarrow T/T(a) \qquad \mathrm{and}  \qquad p_{t,x}: \bT^2 \times T \twoheadrightarrow (\bT^2 \times T)/T(t,x)
\]
for the quotient maps, and let 
\[
\iota_{s_1,s_2}: T \hookrightarrow \bT^2 \times T
\]
denote the inclusion of the fiber over $(s_1,s_2) \in \bT^2$. 
\end{notation}

\begin{remark}\label{fern}
It follows from Lemma \ref{surj} that there is a commutative diagram 
\begin{equation}\label{rect}
\begin{tikzcd}
T(a) \ar[r,hook] \ar[d,hook] & T \ar[r, two heads,"p_a"] \ar[d, "\iota_{0,0}",hook] & T/T(a) \ar[d, "\nu", dashed] \\
T(t,x) \ar[r,hook] \ar[d,two heads] & \bT^2 \times T \ar[r, two heads,"{p_{t,x}}"] \ar[d,two heads] & (\bT^2 \times T)/T(t,x) \ar[d] \\
\bT^2 \ar[r,equal] & \bT^2 \ar[r] & 0
\end{tikzcd}
\end{equation}
where $\nu$ is induced by $\iota_{0,0}$.
\end{remark}

\begin{lemma}\label{groupss} 
The map $\nu$ of diagram \eqref{rect} is an isomorphism.
\end{lemma}

\begin{proof}
By Lemma \ref{surj}, the left hand column is a short exact sequence. It is clear that the middle column is a short exact sequence, as are all rows. It now follows from a standard diagram chase that $\nu$ is an isomorphism.
\end{proof}

\begin{remark}\label{pushy}
The description of $T(t,x)$ in Remark \ref{push} allows us to explicitly describe $L^2X^{t,x} := L^2X^{T(t,x)}$. By this description, a loop $\gamma \in L^2X$ is fixed by $T(t,x)$ if and only if
\[
\gamma(s_1,s_2) = \exp_{\bT^2 \times T}((r_1,0,x_1 r_1)) \cdot \gamma(s_1,s_2) = \exp_{T}( x_1 r_1) \cdot \gamma(s_1 - \exp(r_1),s_2)
\]
and
\[
\gamma(s_1,s_2) = \exp_{\bT^2 \times T}((0,r_2,x_2 r_2)) \cdot \gamma(s_1,s_2) = \exp_T(x_2 r_2) \cdot \gamma(s_1,s_2- \exp( r_2))
\]
for all $s=(s_1,s_2) \in \bT^2$ and all $r_1,r_2 \in \R$. By setting $s_1 = s_2 = 0$, one sees that this is true if and only if 
\[
\gamma(s_1,s_2) = \exp_T(x_1 s_1 + x_2s_2) \cdot \gamma(0,0),
\]
where $\gamma(0,0) \in X^a$ for $a = \zeta_T(t,x)$. Therefore, we have
\[
L^2X^{t,x} = \{ \gamma \in L^2X \, | \,  \gamma(s_1,s_2) = \exp_T(x_1 s_1 +  x_2 s_2)\cdot z, \, z \in X^a\}.
\]
Note that, although $s_1,s_2$ are elements of $\bT = \R/\Z$, the loop $\exp_T(x_1 s_1 +  x_2 s_2)\cdot z$ is well defined since $\exp_T$ is a homomorphism and $z$ is fixed by both $\exp_T(x_1)$ and $\exp_T(x_2)$. Note also that the image of each $\gamma \in L^2X^{t,x}$ is contained in a single $T$-orbit, so that $L^2X^{t,x}$ is contained in the space $L^2X^{gh}$ of \textit{ghost maps}, which appears in Section 5 of \cite{Rezk}. This description of $L^2X^{t,x}$ should also be compared to Corollary 3.4 in \cite{Kitch1}.
\end{remark}

\begin{remark}
The map 
\[
\begin{tikzcd}
ev: L^2X \ar[r]& X
\end{tikzcd}
\]
given by evaluating a loop at $(0,0)$ is evidently $T$-equivariant and continuous.
\end{remark}

\begin{lemma}\label{equal}
The map $ev$ induces a homeomorphism
\[
\begin{tikzcd}
ev_{t,x}: L^2X^{t,x} \ar[r,"\sim"]& X^{a}
\end{tikzcd}
\]
which is natural in $X$, and equivariant with respect to $\nu$. 
\end{lemma}

\begin{proof}
This is evident by Remark \ref{pushy}.
\end{proof}

\begin{example}\label{loopex}
We calculate $L^2X^{t,x}$ in the example of the representation sphere $X = S_\lambda$ associated to $\lambda \in \hat{T} \cong \Z$. For $a \in E_{T,t}$, let $a_1,a_2$ denote the preimage of $a$ under $\chi_{T,t}$. Then 
\[
T(a) = \langle a_1,a_2\rangle \subset T,
\]
and we have
\[
X^a =  \begin{cases} \{N,S\}  &\mbox{if } (a_1,a_2) \neq (0,0) \\ 
S_\lambda & \mbox{if } (a_1,a_2) = (0,0)\end{cases}. 
\]
Let $x \in \zeta_{T,t}^{-1}(a)$ and write $x = x_1t_1 + x_2t_2$. We have
\[
L^2X^{t,x} = \begin{cases} \{N,S\}  &\mbox{if } (x_1,x_2) \notin \Z^2 \\ 
\{ \exp_T(x_1s_1+x_2s_2)\cdot z \, | \,z \in S_\lambda \} & \mbox{if } (x_1,x_2) \in \Z^2 \end{cases}
\]
where $N,S$ denote constant loops. A double loop $\exp_T(x_1s_1+x_2s_2)\cdot_\lambda z$, for $z \neq N,S$, wraps the first loop around the sphere $\lambda x_1$ times, and the second loop around $\lambda x_2$ times, where the loops run parallel to the equator. The direction of the loop corresponds to the sign of the $\lambda x_i$. 
\end{example}

\begin{lemma}\label{hard}
Let $\U$ be an open cover of $T \times T$ adapted to $\S(X)$. Let $(t,x),(t',y) \in \X^+ \times \t_\C$, with $a = \zeta_T(t,x)$ and $b = \zeta_T(t',y)$. We have the following.
\begin{enumerate}
\item If $V_{t,x} \cap V_{t',y} \neq \emptyset$, then either $X^b \subset X^a$ or $X^a \subset X^b$. 
\item If $V_{t,x} \cap V_{t',y} \neq \emptyset$ and $X^b \subset X^a$, then $L^2X^{t',y} \subset L^2X^{t,x}$.
\end{enumerate}
\end{lemma}

\begin{proof}
Write $x = x_1t_1 + x_2t_2$, $y = y_1t_1' + y_2t_2'$ and $a_i = \exp_T(x_i)$, $b_i = \exp_T(y_i)$. Since $V_{t,x} \cap V_{t',y} \neq \emptyset$, we have
\[
V_{x_1,x_2} \cap V_{y_1,y_2} \neq \emptyset
\]
by definition. This implies that 
\[
U_{a_1,a_2} \cap U_{b_1,b_2} \neq \emptyset,
\]
so by the second property of an adapted cover, either $(a_1,a_2) \leq (b_1,b_2)$ or $(b_1,b_2) \leq (a_1,a_2)$. This implies that either $X^b \subset X^a$ or $X^a \subset X^b$, which yields the first part of the result.

For the second part, assume that $X^b \subset X^a$, and let $\gamma \in L^2X^{t',y}$. Let $z = \gamma(0,0)$. By the description in Remark \ref{pushy}, we have
\[
\gamma(s_1,s_2) = \exp_T(y_1s_1 + y_2 s_2)\cdot z.
\]
Let $K \subset T$ be the isotropy group of $z \in X^b \subset X^a$, so that $K \in \S(X)$ and $a_i,b_i \in K$ for $i = 1,2$. The condition 
\[
V_{x_1,x_2} \cap V_{y_1,y_2} \neq \emptyset
\]
implies by Lemma \ref{hard3} that $(x_1,x_2)$ and $(y_1,y_2)$ lie in the same component of $\exp_{T\times T}^{-1}(K\times K)$, since $\U$ is adapted to $\S(X)$. Therefore, $(x_1 - y_1, x_2 - y_2)$ lies in the identity component of $\exp_{T\times T}^{-1}(K\times K)$, which is equal to $\mathrm{Lie}(K) \times \mathrm{Lie}(K)$. This implies that $z$ is fixed by $\exp_T((x_1-y_1)r_1)$ and $\exp_T((x_2-y_2)r_2)$ for all $r_1,r_2 \in \R$.

We can now write
\[
\begin{array}{rcl}
\gamma(s_1,s_2) &=& \exp_T( y_1 s_1 +  y_2 s_2)\cdot z \\
&=& \exp_T(y_1 s_1+  y_2 s_2) \cdot (\exp_T((x_1 - y_1)s_1 + (x_2 - y_2)s_2)\cdot z) \\
&=& \exp_T(x_1 s_1 + x_2 s_2)\cdot z,
\end{array}
\]
which is a loop in $L^2X^{t,x}$ since $z \in X^a$. This yields the second part of the result.
\end{proof}

\begin{lemma}\label{hay}
Let $\U$ be an open cover adapted to $\S(X)$. If $(t',y) \in V_{t,x}$, then $L^2X^{t',y} \subset L^2X^{t,x}$.
\end{lemma}

\begin{proof}
Write $a = \zeta_T(t,x)$, $b = \zeta_T(t',y)$, $x = x_1t_1 + x_2t_2$, $y = y_1t_1' + y_2t_2'$, $b_i = \exp_T(y_i)$, and $a_i = \exp_T(x_i)$. Since $(t',y) \in V_{t,x}$, we have 
\[
(y_1,y_2) \in V_{x_1,x_2} \subset \t \times \t
\] 
and so 
\[
(b_1,b_2) \in U_{a_1,a_2} \subset T\times T.
\]
Therefore $(a_1,a_2) \leq (b_1,b_2)$ by Lemma \ref{hard2}, from which it follows that $X^b \subset X^a$, since $\U$ is adapted to $\S(X)$. Lemma \ref{hard} yields the result.
\end{proof}

\begin{example}\label{incex}
We examine the inclusions of Lemmas \ref{hard} and \ref{hay} in the example of the representation sphere $X = S_\lambda$. Let $(t,x), (t',y) \in \X^+ \times \t_\C$, write $x = x_1t_1+x_2t_2$, $y=y_1t_1' + y_2t_2'$, $a = \zeta_T(t,x)$ and $b = \zeta_T(t',y)$. If $V_{t,x} \cap V_{t,y} = \emptyset$, then by Definition \ref{defo} we have
\[
V_{x_1,x_2} \cap V_{y_1,y_2} \neq \emptyset.
\]
Therefore, by Example \ref{coverex}, at least one of $(x_1,x_2)$ and $(y_1,y_2)$ lies outside the lattice $\Z^2 \subset \R^2$. By Example \ref{loopex}, this means that either $L^2X^{t',y} = X^b = \{N,S\}$ or $L^2X^{t,x} = X^a = \{N,S\}$, and we clearly have either $X^b \subset X^a$ or $X^a \subset X^b$. If we assume that $X^b \subset X^a$, then since at least one of the spaces is equal to $\{N,S\}$, we must have $X^b = \{N,S\}$. Therefore, $L^2X^{t',y} \subset L^2X^{t,x}$. Note that if we had the additional hypothesis that $(t',y) \in V_{t,x}$, then this would imply that $(y_1,y_2) \in V_{x_1,x_2}$, and by the description in Example \ref{coverex} we would have $(y_1,y_2) \notin \Z^2$, which means that $X^b \subset X^a$. So, with this hypothesis, no assumption would be necessary.
\end{example}

\begin{proposition}\label{extend}
Let $Y \subset L^2X$ be a finite $\bT^2 \times T$-CW subcomplex. The inclusion $Y^{t,x} \subset Y$ induces an isomorphism of $\Z/2\Z$-graded $\O_{V_{t,x}}$-algebras
\[
\H^*_{\bT^2\times T}(Y)_{V_{t,x}} \cong \H^*_{\bT^2\times T}(Y^{t,x})_{V_{t,x}}.
\]
\end{proposition}

\begin{proof}
Let $(t',y) \in V_{t,x}$. Then $L^2X^{t',y} \subset L^2X^{t,x}$ by Lemma \ref{hay}, which implies that $Y^{t',y} \subset Y^{t,x}$. Consider the commutative diagram

\begin{equation}
\begin{tikzcd}
\H^*_{\bT^2\times T}(Y)_{V_{t,x}} \ar[dr] \ar[rr] && \H^*_{\bT^2\times T}(Y^{t,x})_{V_{t,x}} \ar[dl] \\
&\H^*_{\bT^2\times T}(Y^{t',y})_{V_{t,x}}, &
\end{tikzcd}
\end{equation}
\noindent
which is induced by the evident inclusions. Taking stalks at $(t',y)$, Theorem \ref{localisationn} implies that the two diagonal maps are isomorphisms, and so the horizontal map is also an isomorphism. The isomorphism of the proposition follows. 
\end{proof}

\begin{remark} 
The subspace $L^2X^{t,x} \subset L^2X$ is a $\bT^2 \times T$-equivariant CW subcomplex, consisting of those equivariant cells in $L^2X$ whose isotropy group contains $T(t,x)$. In fact, it follows easily from Lemma \ref{equal} that $L^2X^{t,x}$ is a finite $\bT^2 \times T$-CW complex, since $X$ is finite. 
\end{remark}

\begin{corollary}\label{stalkss}
The inclusion $L^2X^{t,x} \subset L^2X$ induces an isomorphism of $\Z/2\Z$-graded $\O_{V_{t,x}}$-algebras
\[
\H^*_{\bT^2 \times T}(L^2X)_{V_{t,x}} \cong \H^*_{\bT^2 \times T}(L^2X^{t,x})_{V_{t,x}},
\]
natural in $X$. In particular, we have an isomorphism of stalks
\[
\H^*_{\bT^2 \times T}(L^2X)_{(t,x)} \cong \H^*_{\bT^2 \times T}(L^2X^{t,x})_{(t,x)}.
\]
\end{corollary}

\begin{proof}
It follows from Definition \ref{varlimm} and Proposition \ref{extend} that 
\[
\begin{array}{rcl}
\H^*_{\bT^2 \times T}(L^2X)_{V_{t,x}} &=& \varprojlim_{Y \subset L^2X} \H^*_{\bT^2 \times T}(Y)_{V_{t,x}} \\
&\cong& \varprojlim_{Y \subset L^2X} \H^*_{\bT^2 \times T}(Y^{t,x})_{V_{t,x}} \\
&=& \H^*_{\bT^2 \times T}(L^2X^{t,x})_{V_{t,x}}.
\end{array}
\]
The final equality holds by definition of the inverse limit, since each $Y^{t,x}$ is contained in the finite $\bT^2 \times T$-CW subcomplex $L^2X^{t,x} \subset L^2X$. One shows naturality with respect to a $T$-equivariant map $f: X \to Y$ by refining the cover $\U$ so that it is adapted to $\S(f)$, and by using the functoriality of the loop space functor and of Borel-equivariant cohomology. The second statement follows immediately by definition of the stalk.
\end{proof}

\begin{remark}\label{coherent} 
It follows from Corollary \ref{stalkss} that $\H^*_{\bT^2 \times T}(L^2X)$ is a coherent sheaf of $\O_{\X^+ \times \t_\C}$-modules, since $L^2X^{t,x}$ is a finite $\bT^2 \times T$-CW complex.
\end{remark}

\section{The construction of the equivariant sheaf $\E^*_T(X)$ over $E_T$}\label{brew}

In this section, we begin by showing that $\H^*_{\bT^2 \times T}(L^2X)$ depends only on loops contained in the subspace
\[
\bigcup_{(t,x)\in \X^+ \times \t_\C} L^2X^{t,x} 
\]
of $L^2X$. This is an important feature of our construction which will make computations much more tractable. We then show that the action of the extended double loop group $\widetilde{L^2T}$ on $L^2X$, which was defined in Remark \ref{wei1}, induces an action of the Weyl group 
\[
W_{\widetilde{L^2T}} = \SL_2(\Z) \ltimes \check{T}^2
\]
on the sheaf $\H^*_{\bT^2 \times T}(L^2X)$, which also carries a natural $\C^\times$-action. Finally, we define the $\C^\times \times \SL_2(\Z)$-equivariant sheaf $\E^*_T(X)$ over $E_T$ as the $\check{T}^2$-invariants of the pushforward of $\H^*_{\bT^2 \times T}(L^2X)$ along $\zeta_T$. In the next section we will show that this is a $T$-equivariant elliptic cohomology theory in an appropriate sense. 

\begin{definition}
Let $X$ be a finite $T$-CW complex, and let $\D(X)$ denote the set of finite $\bT^2 \times T$-CW subcomplexes of $L^2X$ generated by 
\[
\{L^2X^{t,x}\}_{(t,x) \in \X^+ \times \t_\C}
\]
under finite unions and intersections. The set $\D(X)$ is partially ordered by inclusion.
\end{definition}

\begin{theorem}\label{yess}
Let $X$ be a finite $T$-CW complex. There is an isomorphism of $\Z/2\Z$-graded $\O_{\X^+ \times \t_\C}$-algebras
\[
\H^*_{\bT^2 \times T}(L^2X) \cong \varprojlim_{Y \in \D(X)} \H^*_{\bT^2 \times T}(Y)_{\X^+ \times \t_\C}
\]
natural in $X$.
\end{theorem}

\begin{proof}
Consider the union
\[
S := \bigcup_{(t,x) \in \X^+ \times \t_\C} L^2X^{t,x}.
\]
Notice that $S$ is a $\bT^2 \times T$-CW subcomplex of $L^2X$. For each finite $\bT^2 \times T$-CW subcomplex $Y \subset L^2X$, we have an inclusion $Y \cap S \hookrightarrow Y$. The induced map of $\Z/2\Z$-graded $\O_{\X^+ \times \t_\C}$-algebras
\begin{equation}\label{maps1}
\varprojlim_{Y \subset L^2X} \H^*_{\bT^2 \times T}(Y)_{\X^+ \times \t_\C} \to \varprojlim_{Y \subset L^2X} \H^*_{\bT^2 \times T}(Y \cap S)_{\X^+ \times \t_\C}
\end{equation}
is natural in $X$, by the functoriality of Borel-equivariant cohomology. Let $(t,x)$ be an arbitrary point. By Corollary \ref{stalkss}, the map \eqref{maps1} induces an isomorphism of stalks at $(t,x)$ because $S$ contains $L^2X^{t,x}$. Therefore, the map \eqref{maps1} is an isomorphism of sheaves.

It is clear that the set $\{S \cap Y \, | \, Y \subset L^2X \, \text{finite}\}$ is equal to the set of all finite equivariant subcomplexes of $S$. Therefore, the target of map \eqref{maps1} is equal to the inverse limit
\begin{equation}\label{mapsss1}
\varprojlim_{Y \subset S} \H^*_{\bT^2 \times T}(Y)_{\X^+ \times \t_\C}
\end{equation}
over all finite equivariant subcomplexes $Y \subset S$. Any such $Y$, since it is finite, is contained in the union of finitely many spaces of the form $L^2X^{t,x}$, which is also a finite $\bT^2 \times T$-CW complex. Therefore, by definition of the inverse limit, \eqref{mapsss1} is equal to 
\[
\varprojlim_{Y \in \D(X)} \H^*_{\bT^2 \times T}(Y)_{\X^+ \times \t_\C},
\]
which completes the proof.
\end{proof}

\begin{remark}
Recall from the introduction to the paper that the subspace $L^2X^{gh} \subset L^2X$ of ghost maps consists precisely of those maps whose image is contained in a single $T$-orbit. We showed in Remark \ref{pushy} that $L^2X^{t,x}$ is a subspace of $L^2X^{gh}$ for all $(t,x) \in \X^+ \times \t_\C$. Thus, Theorem \ref{yess} implies in particular that $\H^*_{\bT^2 \times T}(L^2X)$ depends only on the subspace of ghost maps $L^2X^{gh} \subset L^2X$, as does Rezk's construction in Section 5 of \cite{Rezk}. 
\end{remark}

\begin{remark}
Note that Theorem \ref{yess} allows us to define $\H^*_{\bT^2 \times T}(L^2X)$ without the need for a $\bT^2 \times T$-CW complex structure on $L^2X$. 
\end{remark}

\begin{remark}\label{silmaril}
Recall that $\C^\times$ acts on $\X^+ \times \t_\C$ as in Remarks \ref{ulmo} and \ref{melkor}. For a finite $\bT^2 \times T$-CW complex $Y$, we may equip $\H^*_{\bT^2 \times T}(Y)_{\X^+ \times \t_\C}$ with a $\C^\times$-equivariant structure just as in Remark \ref{ainur}. This induces a $\C^\times$-equivariant structure on the inverse limit $\H_{\bT^2 \times T}(L^2X)$. Explicitly, if $c \in H^i(E(\bT^2 \times T) \times_{\bT^2 \times T} Y; \C)$ and $f \in \O_{\X^+ \times \t_\C}(\lambda \cdot U)$ such that $f(\lambda\cdot(t,x)) = \lambda^j f(t,x)$, then the action of $\lambda \in \C^\times$ sends $c \otimes f$ to $\lambda^{i+j} c \otimes f$.
\end{remark}

\begin{theorem}\label{actyon1}
Let $X$ be a finite $T$-CW complex, and recall from Section \ref{water} the extended double loop group 
\[
\widetilde{L^2T} = (\SL_2(\Z) \ltimes \bT^2) \ltimes L^2T
\]
and the Weyl group 
\[
W_{\widetilde{L^2T}} = \SL_2(\Z) \ltimes \check{T}^2
\]
corresponding to the maximal torus $\bT^2 \times T \subset \widetilde{L^2T}$. Recall the action map 
\[
\alpha: W_{\widetilde{L^2T}} \times \X^+ \times \t_\C \to \X^+ \times \t_\C
\]
given by
\[
\alpha((A, m, t,x)) = (At, x + mt)
\]
and let $\pi$ denote the projection map $W_{\widetilde{L^2T}} \times \X^+ \times \t_\C \twoheadrightarrow \X^+ \times \t_\C$.
The action of $\widetilde{L^2T}$ on $L^2X$ induces a $W_{\widetilde{L^2T}}$-equivariant structure
\[
I: \alpha^* \H^*_{\bT^2 \times T}(L^2X) \to \pi^* \H^*_{\bT^2 \times T}(L^2X)
\]
on $\H^*_{\bT^2 \times T}(L^2X)$ which commutes with the $\C^\times$-equivariant structure.
\end{theorem}

\begin{proof}
The group $\widetilde{L^2T}$ acts on $L^2X$ via
\[
(A,t,\gamma(s)) \cdot \gamma'(s) = \gamma(A^{-1}s-At)\cdot \gamma'(A^{-1}s-At).
\]
Define a left action of $\widetilde{L^2T}$ on $E\widetilde{L^2T} \times L^2X$ by  
\[
g\cdot (e,\gamma) = (e\cdot g^{-1},g\cdot \gamma)
\]
for all $g \in \widetilde{L^2T}$ and $(e,\gamma) \in E\widetilde{L^2T} \times L^2X$. Note that the maximal torus $\bT^2 \times T$ of $\widetilde{L^2T}$ acts freely on $E\widetilde{L^2T}$ via the action of $\widetilde{L^2T}$. Therefore, the quotient of $E\widetilde{L^2T} \times L^2X$ by the $\bT^2 \times T$-action is a model for the $\bT^2 \times T$-equivariant Borel construction
\[
(\bT^2 \times T) \backslash (E\widetilde{L^2T} \times L^2X) \cong E(\bT^2 \times T) \times_{\bT^2 \times T} L^2X.
\]
It follows that the action of $\widetilde{L^2T}$ on $E\widetilde{L^2T} \times L^2X$ induces an action of the Weyl group $W_{\widetilde{L^2T}}$ on 
\[ 
E(\bT^2 \times T) \times_{\bT^2 \times T} L^2X.
\]
On the subspace of fixed loops $L^2X^{t,x}$, the action of $(A,m) \in W_{\widetilde{L^2T}}$ induces a homeomorphism  
\[
E(\bT^2 \times T) \times_{\bT^2 \times T} L^2X^{t,x} \longrightarrow E(\bT^2 \times T) \times_{\bT^2 \times T} L^2X^{At,x+mt}
\]
for each $(t,x) \in \X^+ \times \t_\C$. Writing $x = x_1t_1 + x_2t_2$, the homeomorphism sends
\begin{equation}\label{earl1}
(e,\exp_{T}(x_1s_1 + x_2s_2)\cdot z) \mapsto (e\cdot (A,m)^{-1},\exp_{T}((x_1+m_1,x_2+m_2)A^{-1}(s_1,s_2))\cdot z).
\end{equation}
Note that  
\[
\exp_{T}((x_1+m_1,x_2+m_2)A^{-1}(s_1,s_2))\cdot z 
\]
does in fact lie in $L^2X^{At,x+mt}$, since
\[
x + mt = ((x_1+m_1)t_1, (x_2+m_2)t_2) = (x_1+m_1,x_2+m_2)A^{-1} A(t_1,t_2).
\]
There is an induced isomorphism on cohomology rings
\[
H^*_{\bT^2 \times T}(L^2X^{At,x+mt};\C) \longrightarrow H^*_{\bT^2 \times T}(L^2X^{t,x};\C)
\]
which induces an isomorphism of sheaves
\[
\alpha_{A,m}^* \H^*_{\bT^2 \times T}(L^2X^{At,x+mt})_{\X^+ \times \t_\C} \longrightarrow \pi_{A,m}^*\H^*_{\bT^2\times T}(L^2X^{t,x})_{\X^+ \times \t_\C}.
\]
Here $\alpha_{A,m}$ and $\pi_{A,m}$ are the components of $\alpha$ and $\pi$ corresponding to $(A,m)$. The isomorphism is compatible with the action on the base space because the Chern-Weil isomorphism is natural in $\bT^2 \times T$. In the same way, the action of $\widetilde{L^2T}$ on $L^2X$ induces an isomorphism 
\[
\alpha_{A,m}^* \H^*_{\bT^2 \times T}((A,m)\cdot Y)_{\X^+ \times \t_\C} \longrightarrow \pi^*_{A,m}\H^*_{\bT^2\times T}(Y)_{\X^+ \times \t_\C} 
\]
for each $Y \in \D(X)$, which yields a family of isomorphisms indexed over $\D(X)$, compatible with the inclusion of subcomplexes in $\D(X)$. By Theorem \ref{yess}, we therefore have an induced isomorphism
\[
I_{A,m}: \alpha_{A,m}^*\H^*_{\bT^2 \times T}(L^2X) \longrightarrow \pi^*_{A,m} \H^*_{\bT^2 \times T}(L^2X)
\]
of inverse limit sheaves. It is straightforward to verify that the union $I = \cup I_{A,m}$ of such isomorphisms satisfies the cocycle condition of Remark \ref{ainur} and commutes with the action of $\C^\times$.
\end{proof}

\begin{remark}
We will use $\iota_{t}$ to denote the inclusion of the fiber $\t_\C \hookrightarrow \X^+ \times \t_\C$ over $t \in \X^+$, relying on context to avoid confusion with the usage already defined in Notation \ref{nerf}. We have a commutative diagram of complex manifolds
\[
\begin{tikzcd}
\t_\C \ar[r,hook,"{\iota_t}"] \ar[d,two heads, "{\zeta_{T,t}}"] & \X^+ \times \t_\C \ar[d,two heads, "{\zeta_T}"] \\
E_{T,t} \ar[r,hook] & E_T. 
\end{tikzcd}
\]
The $\C^\times \times \SL_2(\Z)$-action on $E_T$ does not preserve the fiber over $t$. 
\end{remark}

The following definition is analogous to Definition 4.1 in Kitchloo's paper \cite{Kitch1}.

\begin{definition}\label{sherpa}
Let $X$ be a finite $T$-CW complex. We define the $\C^\times \times \SL_2(\Z)$-equivariant, coherent sheaf of $\Z/2\Z$-graded $\O_{E_T}$-algebras 
\[
\E^*_T(X) := ((\zeta_{T})_*\, \H^*_{\bT^2 \times T}(L^2X))^{\check{T}^2}.
\]
We also define the coherent sheaf of $\Z/2\Z$-graded $\O_{E_{T,t}}$-algebras 
\[
\E^*_{T,t}(X) :=  ((\zeta_{T,t})_*\, \iota_t^*\, \H^*_{\bT^2 \times T}(L^2X))^{\check{T}^2}.
\]
\end{definition}

\begin{remark}
By construction, we have that $\E_{T}^*(\pt)$ is equal to the $\C^\times \times \SL_2(\Z)$-equivariant structure sheaf
\[
\O_{E_T} = ((\zeta_T)_*\O_{\X^+ \times \t_\C})^{\check{T}^2} 
\]
of $E_T$. Similarly, $\E_{T,t}^*(\pt)$ is equal to the structure sheaf
\[
\O_{E_{T,t}} = ((\zeta_{T,t})_* \O_{\t_\C})^{\check{T}^2}
\]
of the fiber of $E_T$ over $t \in \X^+$. 
\end{remark}

\begin{example}
It is straightforward to compute $\E_T^*(X)$ when $T = e$ is the trivial group, for this implies that $\D(X) = \{X\}$. We obtain the $\C^\times \times \SL_2(\Z)$-equivariant sheaf whose value on an open subset $U \subset \X^+$ is 
\[
\E_e^*(X)(U) = H_e^*(X) \otimes_\C \O_{\X^+}(U).
\]
\end{example}

\begin{example}\label{modf}
In the case that $T = e$ and $X = \pt$ we obtain the $\C^\times \times \SL_2(\Z)$-equivariant structure sheaf $\O_{\X^+}$. 
\end{example}

\begin{remark}\label{modforms}
The $\C^\times \times \SL_2(\Z)$-action on $\E^*_T(X)$ allows one to extract a $\Z$-graded theory $Ell^*_T(X)$ such that $Ell^*_e(\pt)$ is the ring of weak modular forms. Consider the quotient map
\[
\kappa: E_T \twoheadrightarrow \C^\times \backslash E_T.
\]
The mapping $(t_1,t_2,x) \mapsto (t_1/t_2,x/t_2)$ induces an isomorphism of the quotient with 
\[
\check{T}^2 \backslash (\mathfrak{h} \times \t_\C)
\]
where $\mathfrak{h}$ is the upper half plane of complex numbers, and $m \in \check{T}^2$ acts by $m\cdot (\tau,x) = (\tau, x+m_1\tau + m_2)$. There is a residual action of $\SL_2(\Z)$ on the quotient space given by 
\[
\begin{pmatrix} a & b \\ c & d \end{pmatrix} \cdot (\tau,[x]) = \left(\frac{a\tau + b}{c\tau + d}, \frac{[x]}{c\tau + d}\right).
\]
For $k \in \Z$ and $\lambda \in \C^\times$, let $\Ell^k_T(X)$ denote the eigenspace of $\lambda^k$ in $\E^*_T(X)$. The sheaf $\Ell^k_T(X)$ is preserved by the action of $\SL_2(\Z)$ since $\C^\times$ commutes with $\SL_2(\Z)$. The pushforward of $\Ell^k_T(X)$ along $\kappa$ is also $\SL_2(\Z)$-equivariant, and we write 
\[
Ell^k_T(X) := (\Gamma \, \kappa_* \, \Ell^k_T(X))^{\SL_2(\Z)}
\]
for the vector space of $\SL_2(\Z)$-invariant global sections. We can now define the $\Z$-graded ring
\[
Ell^*_T(X) := \bigoplus_{k\in \Z} Ell^k_T(X).
\]
The map $X \to \pt$ induces a $Ell^*_T(\pt)$-algebra structure on $Ell^*_T(X)$ corresponding to the map of rings $Ell^*_T(\pt) \to Ell^*_T(X)$. By definition, the value of $Ell^{-2k}_T(\pt)$ is the space of holomorphic functions $f$ on $E_T$ satisfying the transformation property
\[
f(t_1,t_2,[x]) = f(at_1 + bt_2, ct_1 + dt_2, [x]) = (ct_1 + dt_2)^{-k} f(\frac{at_1 + bt_2}{ct_1 + dt_2}, 1, \frac{[x]}{ct_1 + dt_2}) 
\]
for all $\begin{pmatrix} a & b \\ c & d \end{pmatrix} \in \SL_2(\Z)$. The first equality holds by $\SL_2(\Z)$-invariance, and the second equality holds since $f \in Ell^{-2k}_T(\pt)$ means that $f(\lambda^2t,\lambda^2[x]) = \lambda^{-2k} f(t,x)$, by the last sentence of Remark \ref{silmaril}. Such a function $f$ determines and is determined by a function $h \in \O(\C^\times \backslash E_T)$ satisfying
\[
(c\tau + d)^k h(\tau,[y]) = h(\frac{a\tau + b}{c\tau + d}, \frac{[y]}{c\tau + d}),
\]
via the substitutions $\tau = t_1/t_2$ and $y = x/t_2$. In particular, $Ell^{-2k}_e(\pt)$ is exactly the space of weak modular forms of weight $k$.
\end{remark}

\section{A torus-equivariant elliptic cohomology theory}\label{salt}

In Definition \ref{sherpa} we defined a functor $\E^*_{T,t}$ on the category of finite $T$-CW complexes, and taking values in the category of $\Z/2\Z$-graded, coherent $\O_{E_{T,t}}$-algebras. Recall from Definition \ref{tear} that a suspension isomorphism for the reduced version $\widetilde{\E}^*_{T,t}$ of $\E^*_{T,t}$ is an isomorphism $\widetilde{\E}^{*+1}_{T,t}(S^1\wedge X) \cong \widetilde{\E}^{*}_{T,t}(X)$, which is natural in $X$ and where the $T$-action on $S^1$ is trivial. In this section, we construct such a map for all $t \in \X^+$ at once, using the description of $\H^*_{\bT^2 \times T}(L^2X)$ given in Theorem \ref{yess} and the suspension isomorphism of cohomology. Then, we give the main result of this section, which is that $\widetilde{\E}^*_{T,t}$ is a reduced $T$-equivariant elliptic cohomology theory in the sense of Definition \ref{tear}.

\begin{remark}
Let $X$ be a pointed finite $T$-CW complex. We regard the loop space $L^2(S^1 \wedge X)$ as a pointed $\bT^2 \times T$-CW complex with basepoint given by the loop $\gamma_\pt: \bT^2 \to \pt \hookrightarrow S^1 \wedge X$. Since the basepoint of $S^1 \wedge X$ is fixed by $T$, the loop $\gamma_\pt$ is fixed by $\bT^2 \times T$, and so $\gamma_\pt$ is contained in $L^2(S^1 \wedge X)^{t,x}$ for all $(t,x)$. Therefore, each $Y \in \D(S^1 \wedge X)$ is a pointed subcomplex of $L^2(S^1 \wedge X)$.
\end{remark}

\begin{lemma}\label{wedger}
For each $(t,x) \in \X^+ \times \t_\C$, we have an equality
\[
L^2(S^1\wedge X)^{t,x} = S^1 \wedge L^2X^{t,x}
\]
as subsets of $L^2(S^1 \wedge X)$.
\end{lemma}

\begin{proof}
Suppose that $\gamma$ is a loop in $L^2(S^1\times X)^{t,x}$ sending $s \mapsto (\gamma_1(s),\gamma_2(s))$. Write $\gamma(0) = (z_1,z_2)$ and $x = x_1t_1 + x_2t_2 \in \t_\C$. The loop $\gamma$ is fixed by $T(t,x)$ if and only if 
\[
(\gamma_1(s),\gamma_2(s)) = \exp_{T\times T}(x_1r_1+x_2r_2)\cdot (\gamma_1(s-r),\gamma_2(s-r)) =  (\gamma_1(s-r),\exp_{T\times T}(x_1r_1+x_2r_2)\cdot \gamma_2(s-r))
\]
for all $r,s \in \bT^2$, since $S^1$ is fixed by $T$. Setting $r = s$, one sees that this is true if and only if
\[
(\gamma_1(r),\gamma_2(r)) =  (\gamma_1(0),\exp_{T\times T}(x_1r_1+x_2r_2)\cdot \gamma_2(0)), 
\]
which holds if and only $\gamma_1$ is constant and $\gamma_2$ is in $L^2X^{t,x}$. Therefore, we have an equality
\[
L^2(S^1 \times X)^{t,x} = S^1 \times L^2X^{t,x}.
\]
To prove the equality of the lemma, consider that the image of $\gamma$ is contained in $0 \times X \cup S^1 \times \pt$ if and only if we have either $\gamma_1(s) = 0$ for all $s$, or $\gamma_2(s) = \pt$ for all $s$ (since $\gamma_1$ is constant). This is the same as saying that $\gamma$ lies in the subset $0 \times L^2X^{t,x} \cup S^1 \times \pt$ of the right hand side. This completes the proof.
\end{proof}

\begin{definition}
Let $X$ be a pointed, finite $T$-CW complex. The value of the reduced theory $\widetilde{\E}_T^*(X)$ on $X$ is defined as the kernel
\[
\widetilde{\E}_T^*(X) := \ker(\E_T^*(X) \to \E_T^*(\pt))
\]
of the map induced by the inclusion of the basepoint. Thus, $\widetilde{\E}_T^*(X)$ is an ideal sheaf of $\Z/2\Z$-graded, coherent $\O_{E_T}$-algebras. 
\end{definition}

\begin{proposition}\label{jezza}
Let $X$ be a finite $T$-CW complex. There is an $\SL_2(\Z)$-equivariant isomorphism of coherent $\O_{E_T}$-algebras
\[
\widetilde{\E}^{*}_T(X) \cong \widetilde{\E}^{*+1}_T(S^1 \wedge X)
\]
natural in $X$. 
\end{proposition}

\begin{proof}
We have the composite isomorphism
\[
\begin{array}{rcl}
\widetilde{\E}^{*}_T(X) &:=&  \ker(\E^*_T(X) \to \E^*_T(\pt)) \\
&\cong& (\zeta_{T*} (\varprojlim_{Y \in \D(X)} \ker(\H^{*}_{\bT^2 \times T}(Y) \to \H^{*}_{\bT^2 \times T}(\pt))_{\X^+ \times \t_\C}))^{\check{T}^2} \\
&\cong& (\zeta_{T*} (\varprojlim_{Y \in \D(X)} \widetilde{\H}^{*}_{\bT^2 \times T}(Y)_{\X^+ \times \t_\C}))^{\check{T}^2}\\
&\cong& (\zeta_{T*} (\varprojlim_{Y \in \D(X)} \widetilde{\H}^{*+1}_{\bT^2 \times T}(S^1 \wedge Y)_{\X^+ \times \t_\C}))^{\check{T}^2 }\\
&=& (\zeta_{T*} (\varprojlim_{Y \in \D(S^1 \wedge X)} \widetilde{\H}^{*+1}_{\bT^2 \times T}(Y)_{\X^+ \times \t_\C}))^{\check{T}^2} \\
&\cong& (\zeta_{T*} (\varprojlim_{Y \in \D(S^1 \wedge X)} \ker(\H^{*+1}_{\bT^2 \times T}(Y) \to \H^{*+1}_{\bT^2 \times T}(\pt))_{\X^+ \times \t_\C}))^{\check{T}^2} \\
&\cong&  \ker(\E^{*+1}_T(S^1\wedge X) \to \E^{*+1}_T(\pt)) \\
&=:&  \widetilde{\E}^{*+1}_T(S^1 \wedge X).
\end{array}
\]
Indeed, the second and seventh lines hold since the inverse limit is a right adjoint functor, and therefore respects all limits, including kernels. The third and sixth lines are a standard characterisation of reduced cohomology. The fourth line is induced by the natural suspension isomorphism on cohomology. The fifth holds because the equality of Lemma \ref{wedger} is an equality of subsets of $L^2(S^1 \wedge X)$ and therefore compatible with inclusion, so that it extends to an equality of partially ordered sets
\[
\D(S^1 \wedge X) = \{S^1 \wedge Y\}_{Y\in \D(X)}.
\]
It is straightforward to show that the isomorphisms are natural in $X$ and $\SL_2(\Z)$-equivariant.
\end{proof}

\begin{corollary}\label{hewson}
The functor $\widetilde{\E}^*_{T,t}$ is a reduced $T$-equivariant elliptic cohomology theory in the sense of Definition \ref{tear}, once it is equipped with the restriction of the suspension isomorphism of Proposition \ref{jezza} to $E_{T,t}$. Furthermore, let $\psi$ be a homomorphism of compact tori $\psi: T \to G$, and let $R_\psi: E_{T,t} \to E_{G,t}$ be the map induced by $\psi$.
Then there exists a natural transformation of functors $\widetilde{\E}^*_{G,t} \to R_{\psi *}\widetilde{\E}^*_{T,t}$.
\end{corollary}

\begin{proof}
It follows from Remark \ref{coherent} that, for a pointed finite $T$-complex $X$, the assignment $X \mapsto \widetilde{\E}^*_{T,t}(X)$ takes values in the category $\Coh(E_{T,t})$ of $\Z/2\Z$-graded, coherent $\O_{E_{T,t}}$-algebras. Furthermore, it is functorial and homotopy invariant, since the double loop space functor and Borel-equivariant cohomology are functorial and homotopy invariant. It is also exact and additive since these properties may be checked on stalks, and on stalks they are inherited from Borel-equivariant cohomology by Corollary \ref{stalkss} (since $L^2X^{t,x} \cong X^a$ by Lemma \ref{equal}; compare the proof of Proposition \ref{kerry}). 

It remains to show naturality in the compact torus $T$. The map $\psi:T \to G$ induces a map $U_\psi: \X^+ \times \Lie(T)_\C \to \X^+ \times \Lie(G)_\C$ of complex manifolds, which allows us to pull back holomorphic functions. Moreover, given a $G$-complex $X$, we may regard $X$ as a $T$-complex via $\psi$, and $\psi$ yields a map 
\[
L^2X \times_{\bT^2 \times T} E(\bT^2 \times T) \to L^2X \times_{\bT^2 \times G} E(\bT^2 \times G)
\]
of topological spaces. This induces a map on cohomology $H^*_{\bT^2 \times G}(L^2X) \to H^*_{\bT^2 \times T}(L^2X)$ which, for $X = \pt$, is compatible with pullback of functions, because the Chern-Weil isomorphism is natural in $T$. Therefore, we have an induced map
\[
H^*_{\bT^2 \times G}(Y) \otimes_{H_{\bT^2 \times G}} \O_{\X^+ \times \Lie(G)_\C} \to U_{\psi *} (H^*_{\bT^2 \times T}(Y) \otimes_{H_{\bT^2 \times T}} \O_{\X^+ \times \Lie(T)_\C}).
\]
Taking the inverse limit over $Y \in \D(X)$ yields a map 
\[
\H^*_{L^2G}(L^2X) \to U_{\psi *} \H^*_{\bT^2 \times T}(L^2X).
\]
Finally, pushing forward along $\zeta_{G,t}$ and taking $\check{G}^2$-invariants yields a map
\[
\E^*_G(X) \to R_{\psi *} \E^*_T(X),
\]
natural in $X$, which restricts to the desired map.
\end{proof}

\section{A calculation of $\E_T^*(T/K)$}\label{pepper}

In this section, we use Theorem \ref{yess} to compute the value of $\E^*_T$ on a single orbit $T/K$, which has a $T$-action induced by the group structure. This allows one to compute $\E^*_T(X)$ for any other finite $T$-CW complex $X$ by using the Mayer-Vietoris sequence. 

Let $i$ denote the inclusion $E_K \hookrightarrow E_T$ of complex manifolds induced by the inclusion $K \subset T$. We will show that there is a canonical isomorphism
\[
\E^*_T(T/K) \cong i_* \O_{E_K}
\]
where the left hand side is a $\Z/2\Z$-graded algebra concentrated in even degree.

We begin by calculating $\D(T/K)$. Let $a \in E_{T,t}$ and write $(a_1,a_2)$ for the preimage of $a$ under $\chi_{T,t}$. We have
\[
(T/K)^a = (T/K)^{\langle a_1,a_2 \rangle} = \begin{cases} T/K &\mbox{if } (a_1,a_2) \in K \times K \\ \emptyset & \mbox{otherwise.} \end{cases}
\]
Let $x \in \zeta_{T,t}^{-1}(a)$ and write $x = x_1t_1 + x_2t_2$. Recall that $\zeta_{T,t}$ denotes the quotient map $\t_\C \twoheadrightarrow E_{T,t}$, and $\exp_T: \t \to T$ denotes the exponential map of $T$. We have that
\begin{equation}\label{henry}
L^2X^{t,x} = \{\gamma(s_1,s_2) = \exp_T(x_1s_1 + x_2s_2) \cdot z \, | \, z\in T/K\}
\end{equation}
if $(x_1,x_2) \in \exp_{T\times T}^{-1}(K \times K)$, and $L^2X^{t,x}$ is empty otherwise. To calculate $\D(T/K)$, we need to calculate the intersections of these subspaces. Suppose 
\[
\gamma(s_1,s_2) \in L^2X^{t,x}  \cap L^2X^{t',x'} 
\]
with $x = x_1t_1 + x_2t_2$ and $x' = x_1't_1'+ x_2't_2'$. Then
\[
\gamma(s_1,s_2) = \exp_T(x_1s_1+x_2s_2)\cdot z = \exp_T(x_1's_1+x_2's_2)\cdot z'
\]
for $z,z' \in T/K$. A straightforward calculation shows that this holds if and only if
\[
z - z' = \exp_T((x_1-x_1')s_1 + (x_2-x_2')s_2) \text{ mod } K
\]
for all $s_1,s_2$, which holds if and only if
\[
z = z' \quad \text{and} \quad (x_1-x_1',x_2-x_2') \in \mathrm{Lie}(K) \times \mathrm{Lie}(K).
\]
Thus, for all $(t,x),(t',x') \in \X^+ \times \t_\C$, the intersection
\[
L^2X^{t,x} \cap L^2X^{t',x'} = L^2X^{t,x} = L^2X^{t',x'}, 
\]
if and only if $(x_1,x_2),(x_1',x_2')$ lie in the same component of $\exp_{T\times T}^{-1}(K \times K)$, and is empty otherwise. In particular, the spaces $L^2X^{t,x}$ are indexed by the components of $\exp_{T\times T}^{-1}(K \times K)$, where the component containing $(x_1,x_2)$ has the form
\begin{equation}\label{goop}
\{(x_1,x_2)\} + \mathrm{Lie}(K) \times \mathrm{Lie}(K) \subset \t \times \t.
\end{equation}
We will call $(x_1,x_2)$ a \textit{representative} of this component. Moreover, we have seen that there are no nonempty intersections between two spaces of the form $L^2X^{t,x}$, unless they are equal. Therefore $\D(X)$ is the set of finite disjoint unions of such spaces. \par

Before moving on to the next step, which is to calculate the cohomology ring $H_{\bT^2 \times T}(L^2X^{t,x})$, we need to know more about the space $L^2X^{t,x}$. Since the $\bT^2 \times T$-action on $L^2X^{t,x}$ is clearly transitive, we can apply the change of groups property of Proposition \ref{change1} if we know the $\bT^2 \times T$-isotropy. An element $(r,u) \in \bT^2 \times T$ fixes a nonempty subset $L^2X^{t,x}$ if and only if
\[
(r,u) \cdot (\exp_T(x_1s_1 + x_2s_2)\cdot z) = (\exp_T(x_1(s_1-r_1) + x_2(s_2-r_2)) + u ) \cdot z
= \exp_T(x_1s_1 + x_2s_2)\cdot z,
\]
which holds if and only if $\exp_T(-x_1r_1 - x_2r_2) + u$ fixes $z$. Therefore, we must have
\[
u - \exp_T(x_1r_1 + x_2r_2) \in K,
\]
which means that the isotropy group of the $\bT^2 \times T$-orbit $L^2X^{t,x}$ is equal to 
\begin{equation}\label{ernie}
\{(r,u) \in \bT^2 \times T \, | \, u - \exp_T(x_1r_1 + x_2r_2) \in K \} = \langle T(t,x),K \rangle.
\end{equation}
Furthermore, since two spaces of the form $L^2X^{t,x}$ are equal if they correspond to the same component of $\exp_{T\times T}^{-1}(K \times K)$, we must have that 
\[
\langle T(t,x),K \rangle = \langle T(t',x'),K \rangle
\]
whenever $(x_1,x_2),(x_1',x_2')$ lie in the same component $\exp_{T\times T}^{-1}(K \times K)$. Using Proposition \ref{change1}, we calculate 
\[
H_{\bT^2 \times T}(L^2X^{t,x}) \cong H_{\bT^2 \times T}((\bT^2 \times T)/\langle T(t,x),K \rangle) \cong H_{\langle T(t,x),K \rangle}.
\]
The value of $\H^*_{\bT^2 \times T}(L^2X)$ on an open subset $U \subset \X^+ \times \t_\C$ is therefore
\begin{gather*}
\varprojlim_{Y \in \D(X)} H_{\bT^2 \times T}(Y) \otimes_{H_{\bT^2 \times T}} \O_{\X^+ \times \t_\C}(U) \\ = \prod_{(x_1,x_2) \in J(K)} H_{\langle T(x_1,x_2),K \rangle}  \otimes_{H_{\bT^2 \times T}} \O_{\X^+ \times \t_\C}(U)
\end{gather*}
where the product is indexed over a set $J(K) = \{(x_1,x_2)\}$ of representatives of the components of $\exp_{T\times T}^{-1}(K \times K)$, and $T(x_1,x_2) := T(t,x)$ for any $(t,x)$ such that $x = x_1t_1+x_2t_2$. \par

From the description in \eqref{ernie}, we see that
\begin{equation}\label{berrol}
\mathrm{Lie}(\langle T(x_1,x_2),K \rangle)_\C = \{ (t,y) \in \X^+ \times \t_\C \, | \, y \in (x_1 + \mathrm{Lie}(K))t_1 + (x_2 + \mathrm{Lie}(K))t_2 \}.
\end{equation}
Let $I(x_1,x_2,K) \subset \C[t_1,t_2,y]$ be the ideal associated to $\mathrm{Lie}(\langle T(x_1,x_2),K \rangle)_\C \subset \C^2 \times \t_\C$, so that
\[
H_{\langle T(x_1,x_2),K \rangle} = \C[t,y]/I(x_1,x_2,K).
\]
By Proposition 2.8 in \cite{Rosu03}, tensoring over $H_{\bT^2 \times T}$ with the ring of holomorphic functions is an exact functor. Therefore, if $\I(x_1,x_2,K) \subset \O_{\X^+ \times \t_\C}$ denotes the analytic ideal associated to $I(x_1,x_2,K)$, we can write
\begin{equation}\label{beer}
\H^*_{\bT^2 \times T}(L^2(T/K)) = \prod_{(x_1,x_2)\in J(K)}  \O_{\X^+ \times \t_\C}/\I(x_1,x_2,K).
\end{equation}
Note that the left hand side, whose grading is induced by the even and odd grading on cohomology, is concentrated in even degree.

\begin{example}
If $T$ has rank one and $K = \Z/n\Z$, then $\mathrm{Lie}(K)$ is trivial and we have
\[
\H^*_{\bT^2 \times T}(L^2(T/K))(U) = \prod_{(x_1,x_2)} \C[t_1,t_2,y]/(y - x_1t_1 - x_2t_2) \otimes_{\C[t,y]} \O_{\X^+ \times \t_\C}(U)
\]
where $(x_1,x_2)$ ranges over $J(\Z/n\Z) = \{(a_1/n, a_2/n) \, | \, a_1,a_2 \in \Z\}$. This is a holomorphic version of the calculation made by Rezk in Example 5.2 of \cite{Rezk}. Note that the left hand side is concentrated in even degree.
\end{example}

It remains to compute the $\check{T}^2$-invariants of \eqref{beer}, which is the sheaf of holomorphic functions on the complex submanifold
\begin{equation}\label{derk}
\coprod_{(x_1,x_2)\in J(K)} \mathrm{Lie}(\langle T(x_1,x_2),K \rangle)_\C \subset \X^+ \times \t_\C.
\end{equation}
The $\check{T}^2$-invariant functions on this space are canonically identified with functions on the $\check{T}^2$-quotient, and we will show that this quotient is exactly $E_K$. Recall from \eqref{goop} that 
\[
\exp_{T\times T}^{-1}(K \times K) = \coprod_{(x_1,x_2)\in J(K)} \{(x_1,x_2)\} + \mathrm{Lie}(K) \times \mathrm{Lie}(K).
\]
It follows from \eqref{berrol} that the right hand side is the preimage of 
\[
\coprod_{(x_1,x_2)\in J(K)} \mathrm{Lie}(\langle T(x_1,x_2),K \rangle)_\C
\]
under the isomorphism $\xi_T: \X^+ \times \t \times \t \rightarrow \X^+ \times \t_\C$, which sends $(t,x_1,x_2)$ to $(t,x_1t_1+x_2t_2)$, as in Remark \ref{grut}. Recall that $\xi_T$ is $\check{T}^2$-equivariant, and that it induces the isomorphism $\chi_T: \X^+ \times T \times T \to E_T$ of orbit spaces, as expressed by diagram \eqref{deer}. We may summarise our present situation by the commutative diagram
\[
\begin{tikzcd}
\X^+ \times \exp_{T\times T}^{-1}(K \times K) \ar[d, twoheadrightarrow] \ar[r,"{\cong}"] & \coprod_{(x_1,x_2)\in J(K)} \mathrm{Lie}(\langle T(x_1,x_2),K \rangle)_\C \ar[d, twoheadrightarrow] \\
\X^+ \times K \times K \ar[r,"{\cong}","{\chi_K}"']& E_K.
\end{tikzcd}
\]
Here, the upper horizontal map, left vertical map and right vertical map are suitable restrictions of the respective maps $\xi_T$, $\id \times \exp_{T\times T}$ and $\zeta_T$ of diagram \eqref{deer}. Thus one sees that $E_K$ is the complex analytic quotient of \eqref{derk} by the $\check{T}^2$-action. In other words, we have a canonical isomorphism
\[
\E^*_T(T/K) \cong i_* \O_{E_K},
\]
where the left hand side is concentrated in even degree.

\section{A local description}\label{calf}

In this section, for $t \in \X^+$ and a cover $\U$ of $E_{T,t}$ which is adapted to $X$, we give a local description of $\E^*_{T,t}(X)$ with respect to $\U$. This description will turn out to be identical to Grojnowski's construction of $\G^*_{T,t}(X)$ over the elliptic curve $E_t$, as outlined in Section \ref{Grojn}. 

Before we can give the description, we need two technical lemmas. Recall that $\iota_s: T \hookrightarrow \bT^2 \times T$ denotes the inclusion of the fiber over $s \in \bT^2$, and recall the maps $p_{t,x}$, $p_a$ and $\nu$ which fit into a commutative diagram 
\begin{equation}\label{rect1}
\begin{tikzcd}
T \ar[r, two heads,"p_a"] \ar[d, "\iota_{0,0}",hook] & T/T(a) \ar[d, "\nu", "\cong"'] \\
 \bT^2 \times T \ar[r, two heads,"{p_{t,x}}"] & (\bT^2 \times T)/T(t,x) \\
\end{tikzcd}
\end{equation}
as in Remark \ref{fern}.

In this section we write $K(t,x)$ for the quotient group $(\bT^2 \times T)/T(t,x)$, we write $\mathrm{tr}_{t,x}$ for translation by $(t,x)$ in $\C^2 \times \t_\C$, and $\mathrm{tr}_x$ for translation by $x$ in $\t_\C$. For any map $f$ of compact Lie groups, we abuse notation and also write $f$ for the induced map of complex Lie algebras (as we have been doing with the symbol $\iota$).

\begin{lemma}\label{dg}
The diagram
\begin{equation}
\begin{tikzcd}
\t_\C \ar[r,"p_{a} \circ \mathrm{tr}_{-x}",two heads] \ar[d,hook,"\iota_{t}"] & \mathrm{Lie}(T/T(a))_\C \ar[d,"\nu","\cong"'] \\
\C^2 \times \t_\C \ar[r,"p_{t,x}", two heads] & \mathrm{Lie}(K(t,x))_\C
\end{tikzcd}
\end{equation}
of maps of complex Lie algebras commutes. 
\end{lemma}

\begin{proof}
The commutative diagram \eqref{rect1} induces a commutative diagram of complex Lie algebras, so that
\[
\nu \circ p_a = p_{t,x} \circ \iota_{0,0}
\]
on complex Lie algebras. Consider the commutative diagram
\[
\begin{tikzcd}
T(t,x) \ar[r] \ar[d] & \bT^2 \times T \ar[d] \\
0 \ar[r] & (\bT^2 \times T)/T(t,x) 
\end{tikzcd}
\]
of compact abelian groups. By applying the Lie algebra functor and then tensoring with $\C$, we see that $(t,x)$ lies in the kernel of 
\[
p_{t,x}: \C^2 \times \t_\C \twoheadrightarrow \mathrm{Lie}((\bT^2 \times T)/T(t,x))_\C.
\]
Therefore,
\[
p_{t,x} = p_{t,x} \circ \mathrm{tr}_{t,x}
\]
and so
\[
\begin{array}{rcl}
\nu \circ p_a &=& p_{t,x} \circ \mathrm{tr}_{t,x} \circ \iota_{0,0} \\
&=& p_{t,x} \circ \iota_t \circ \mathrm{tr}_{x}. \\
\end{array}
\]
Composing on the right by $\mathrm{tr}_{-x}$ now yields the result.
\end{proof}

\begin{lemma}\label{equal1}
There is an isomorphism of sheaves of $\Z/2\Z$-graded $\O_{\t_\C}$-algebras
\[
\mathrm{tr}_{-x}^* \, p_{a}^*\, \H^*_{T/T(a)}(X^{a}) \cong \iota_{t}^* \, p_{t,x}^*\, \H^*_{K(t,x)}(L^2X^{t,x}) 
\]
natural in $X$.
\end{lemma}

\begin{proof}
Recall the isomorphism $\nu: K(t,x) \cong T/T(a)$ of diagram \eqref{rect1}. By Lemma \ref{equal}, the evaluation map
\[
ev_{t,x}: L^2X^{t,x} \cong X^a
\]
is natural in $X$ and equivariant with respect to $\nu$. The induced homeomorphism
\[
E(K(t,x)) \times_{K(t,x)} L^2X^{t,x} \cong E(T/T(a)) \times_{T/T(a)} X^a
\]
induces, in turn, an isomorphism of $\Z/2\Z$-graded $H_{T/T(a)}$-algebras
\[
H^*_{T/T(a)}(X^{a}) \cong H^*_{K(t,x)}(L^2X^{t,x}),
\]
natural in $X$. Here the ring $H_{T/T(a)}$ acts on the target via the isomorphism 
\[
H_{K(t,x)} \cong H_{T/T(a)}
\]
induced by $\nu$. Thus, we have an isomorphism of $\O_{\mathrm{Lie}(T/T(a))_\C}$-algebras   
\[
\H^*_{T/T(a)}(X^{a}) \cong \nu^* \, \H^*_{K(t,x)}(L^2X^{t,x}),
\]
and hence an isomorphism of $\O_{\t_\C}$-algebras,
\[
\mathrm{tr}_{-x}^*\, p_a^*\, \H^*_{T/T(a)}(X^{a}) \cong \mathrm{tr}^*_{-x}\, p_{a}^*\,\nu^* \, \H^*_{K(t,x)}(L^2X^{t,x})
\]
natural in $X$. The result now follows from Lemma \ref{dg}.
\end{proof}

\begin{notation}
In what follows, we shall use the abbreviation
\[
H^*_T(X)_U := H^*_T(X) \otimes_{H_T} \O_{\t_\C}(U)
\]
in order to reduce notational clutter.
\end{notation}

\begin{theorem}\label{iso}
Let $X$ be a finite $T$-CW complex and let $\U$ be a cover adapted to $X$. Let $a \in E_{T,t}$, and let $U_a$ be the corresponding open neighbourhood of $a$ in $E_{T,t}$. There is an isomorphism of sheaves of $\Z/2\Z$-graded $\O_{U_a}$-algebras
\[
\E^*_{T,t}(X)_{U_a} \cong \G^*_{T,t}(X)_{U_a},
\]
natural in $X$.
\end{theorem}

\begin{proof}
Recall the quotient map 
\[
\zeta_{T,t}: \t_\C \twoheadrightarrow \check{T}^2 \backslash \t_\C =: E_{T,t},
\]
where the action of $m \in \check{T}^2$ on $\t_\C$ is given by $x \mapsto x + mt$. The sheaf $\E^*_{T,t}(X)$ is defined as the $\check{T}^2$-invariants of the pushforward of the equivariant sheaf $\iota_t^* \, \H^*_{\bT^2 \times T}(L^2X)$ along this map. For $x \in \zeta_{T,t}^{-1}(a)$, write $V_x$ for the component of $\zeta_{T,t}^{-1}(U_a)$ containing $x$. We have a sequence of isomorphisms
\begin{equation}\label{compos}
\begin{array}{rcl}
((\zeta_{T,t})_* \,  \iota_{t}^* \,\H^*_{\bT^2 \times T}(L^2X))_{U_a} &\cong& \prod_{x \, \in \, \zeta_{T,t}^{-1}(a)} \:  (\iota_{t}^* \, \H^*_{\bT^2\times T}(L^2X^{t,x}))_{V_{x}} \\
&\cong& \prod_{x \, \in \, \zeta_{T,t}^{-1}(a)} \: (\iota_{t}^* \, p_{t,x}^*\, \H^*_{K(t,x)}(L^2X^{t,x}))_{V_x} \\
&\cong& \prod_{x \, \in \, \zeta_{T,t}^{-1}(a)} \: (\mathrm{tr}_{-x}^*\, p_a^*\, \H^*_{T/T(a)}(X^{a}))_{V_x} \\
&\cong& \prod_{x \, \in \, \zeta_{T,t}^{-1}(a)} \: (\mathrm{tr}_{-x}^*\, \H^*_{T}(X^{a}))_{V_x} \\
\end{array}
\end{equation}
The first map is the isomorphism of Corollary \ref{extend}. The second map and fourth maps are the isomorphism of Proposition \ref{changeh}. The third map is the isomorphism of Lemma \ref{equal1}. The first to the fourth maps, in each factor, have been shown to preserve the $\O_{V_{x}}$-algebra structure. The entire composite is therefore an isomorphism of $\O_{U_a}$-algebras, where $f\in \O_{U_a}(U)$ acts on the right hand side via multiplication by $\zeta_{T,t}^*f$. Each map has been shown to be natural in $X$, replacing, where necessary, the covering $\U$ with a refinement $\U(f)$ adapted to a $T$-equivariant map $f:X \to Y$. 
Each map has been shown to preserve the 
$\Z/2\Z$-grading by odd and even elements.

It remains to find the image of the $\check{T}^2$-invariants under the composite map above. Let $U \subset U_a$ be an open subset and let $V \subset V_x$ be an open set such that $V \cong U$ via $\zeta_{T,t}$. Over $U$, the composite is equal to  
\[
\begin{array}{rcl}
\prod \: H_{L^2T}(L^2X)_{\{t\} \times V} &\cong& \prod \: H_{\bT^2\times T}(L^2X^{t,x})_{\{t\} \times V}  \\
&\cong& \prod \: H_{K(t,x)}(L^2X^{t,x})_{\{t\} \times V}  \\
&\cong& \prod \: H_{T/T(a)}(X^{a})_{V - x} \\
&\cong& \prod \: H_{T}(X^{a})_{V-x} \\
\end{array}
\]
where each product runs over all $x \in \zeta_{T,t}^{-1}(a)$. Note that 
\[
V_x - x = V_{x+mt} - x-mt
\]
for all $m \in \check{T}^2$. We will show  
\[
( \prod_{x \, \in \, \zeta_{T,t}^{-1}(a)} H_T(X^a)_{V - x} )^{\check{T}^2} =  H_T(X^a)_{V - x}
\]
by showing that $\check{T}^2$ merely permutes the indexing set of the product. This yields our result via the isomorphism 
\[
\G^*_{T,t}(X)_{U_a}(U) = H^*_T(X^a) \otimes_{H_T} \O_{E_{T,t}}(U - a) \cong H^*_T(X^a) \otimes_{H_T} \O_{\t_\C}(V - x) =: H_{T}(X^{a})_{V - x}
\]
induced by the canonical isomorphism $V - x \cong U - a$, as in Remark \ref{canon}.\par

In other words, we must show that the action of $m \in \check{T}^2$ induces the identity map 
\[
H_{T/T(a)}(X^{a})_{V_{x+mt} - x-mt}  = H_{T/T(a)}(X^{a})_{V_x - x}.
\]
To do this, it suffices to check the commutativity of two diagrams. The first diagram is
\begin{equation}\label{pal}
\begin{tikzcd}
H_{\bT^2 \times T}(L^2X^{t,x+mt})_{\{t\} \times V_{x+mt}} \ar[r] \ar[d,"{m^*}"] & H_{K(t,x+mt)}(L^2X^{t,x+mt})_{\{t\} \times V_{x+mt}} \ar[d,"{m^*}"] \\
H_{\bT^2 \times T}(L^2X^{t,x})_{\{t\} \times V_x}  \ar[r] & H_{K(t,x)}(L^2X^{t,x})_{\{t\} \times V_x}  \\
\end{tikzcd}
\end{equation}
where the vertical arrows are induced by the action of $m \in \check{T}^2$ on the spaces
\[
E(\bT^2 \times T) \times_{\bT^2 \times T} L^2X^{t,x} \longrightarrow E(\bT^2 \times T) \times_{\bT^2 \times T} L^2X^{t,x+mt}
\]
and
\[
E(K(t,x)) \times_{K(t,x)} L^2X^{t,x}  \longrightarrow  E(K(t,x+mt)) \times_{K(t,x+mt)} L^2X^{t,x+mt}.
\]
The horizontal maps are defined, as in the proof of Proposition 2.3.4 of \cite{Chen}, using the Eilenberg-Moore spectral sequences associated to the pullback diagrams
\begin{equation}\label{gary}
\begin{tikzcd}
E(\bT^2 \times T) \times_{\bT^2\times T} L^2X^{t,x} \ar[d] \ar[r] & B(\bT^2 \times T) \ar[d] \\
E(K(t,x)) \times_{K(t,x)} L^2X^{t,x} \ar[r] &  B(K(t,x)) \\
\end{tikzcd}
\end{equation}
and
\begin{equation}\label{steve}
\begin{tikzcd}
E(\bT^2 \times T) \times_{\bT^2\times T} L^2X^{t,x+mt} \ar[r] \ar[d] & B(\bT^2 \times T) \ar[d] \\ 
E(K(t,x+mt)) \times_{K(t,x+mt)} L^2X^{t,x+mt} \ar[r] & B(K(t,x+mt)). \\
\end{tikzcd}
\end{equation}
It is easily verified that the action of $m$ induces an isomorphism from diagram \eqref{gary} to diagram \eqref{steve}, from which it follows that \eqref{pal} commutes. \par

The second diagram to check is
\begin{equation}\label{pall}
\begin{tikzcd}
H_{K(t,x+mt)}(L^2X^{t,x+mt})_{\{t\} \times V_{x+mt}} \ar[d,"{m^*}"] \ar[r] &H_{T/T(a)}(X^{a})_{V_{x+mt} - x-mt} \ar[d,equal] \\
H_{K(t,x)}(L^2X^{t,x})_{\{t\} \times V_x} \ar[r] & H_{T/T(a)}(X^{a})_{V_x - x} \\
\end{tikzcd}
\end{equation}
with vertical maps induced by the action of $m$, and horizontal maps as in the proof of Lemma \ref{equal1}. By the proof of Lemma \ref{equal1}, diagram \eqref{pall} commutes if 
\begin{equation}\label{pal1}
\begin{tikzcd}[row sep=large, column sep=3cm]
E(K(t,x)) \times_{K(t,x)} L^2X^{t,x} \ar[d,"{m}"] \ar[r,"{E\nu^{-1} \times ev_{t,x}}"] \ar[d] & E(T/T(a)) \times_{T/T(a)} X^a \ar[d,equal] \\
E(K(t,x+mt)) \times_{K(t,x+mt)} L^2X^{t,x+mt} \ar[r,"{E\nu^{-1} \times ev_{t,x+mt}}"] & E(T/T(a)) \times_{T/T(a)} X^{a}\\
\end{tikzcd}
\end{equation}
commutes, where $ev_{t,x}: L^2X^{t,x}\cong X^a$ is equivariant with respect to $\nu^{-1}: K(t,x) \cong T/T(a)$. To see that this commutes, note firstly that $\nu$ is induced by the inclusion $T \hookrightarrow \bT^2 \times T$ of the fixed points of the Weyl action $(r,t) \mapsto (r,t+mr)$, which implies that $\nu^{-1} \circ m = \nu^{-1}$. Secondly, note that the action of $m$ on a loop $\gamma$ fixes $\gamma(0,0)$, which implies that $ev_{t,x+mt}\circ m = ev_{t,x}$. These two observations imply the commutativity of diagram \eqref{pal1}, which completes the proof.
\end{proof}

\begin{remark}\label{gluh}
Our aim is now to show that the gluing maps associated to the local description in Theorem \ref{iso} are identical to the gluing maps $\{\phi_{b,a}\}$ of Grojnowski's construction. Before doing this, we introduce a more convenient description of the maps $\{\phi_{b,a}\}$. To this end, let $X$ be a finite $T$-CW complex, let $\U$ be a cover adapted to $\S(X)$, and let $a,b \in E_{T,t}$ be such that $U_a \cap U_b \neq \emptyset$. Choose $x\in \zeta_{T,t}^{-1}(a)$ and $y \in \zeta_{T,t}^{-1}(b)$ such that $V_x \cap V_y \neq \emptyset$, which implies that $V_{t,x} \cap V_{t,y} \neq \emptyset$. It follows from Lemma \ref{hard} that either $X^b \subset X^a$ or $X^a \subset X^b$. We may assume that $X^b \subset X^a$. Let $U$ be an open subset in $U_a \cap U_b$ and let $V \subset V_{x} \cap V_{y}$ be such that $V \cong U$ via $\zeta_{T,t}$. Let $T(a,b)$ be the subgroup of $T$ generated by $T(a)$ and $T(b)$. Consider the composite of isomorphisms
\[
\begin{array}{rcl}
H_T(X^a) \otimes_{H_T} \O_{\t_\C}(V - x) &\cong & H_T(X^b) \otimes_{H_T}  \O_{\t_\C}(V - x) \\
&\cong& H_{T/T(a,b)}(X^b) \otimes_{H_{T/T(a,b)}}  \O_{\t_\C}(V - x) \\ 
&\cong& H_{T/T(a,b)}(X^b) \otimes_{H_{T/T(a,b)}}  \O_{\t_\C}(V - y) \\
&\cong& H_T(X^b) \otimes_{H_T}  \O_{\t_\C}(V - y)
\end{array}
\]
The first map is induced by the inclusion $X^b \subset X^a$, the second map and fourth maps are induced by the change of groups map of Proposition \ref{changeh}, and the third map is  
\[
\id \otimes \mathrm{tr}^*_{y-x}: H_{T/T(a,b)}(X^b) \otimes_{H_{T/T(a,b)}} \O_{\t_\C}(V - x) \longrightarrow H_{T/T(a,b)}(X^b) \otimes_{H_{T/T(a,b)}} \O_{\t_\C}(V - y),
\]
which is a well defined map of $\O_{\t_\C}$-algebras, by Lemma \ref{gart}. All maps thus preserve the $\Z/2\Z$-graded $\O_{\t_\C}$-algebra structure. We now show that the composite map above is identical to the gluing map $\phi_{b,a}$. Consider the commutative diagram
\[
\begin{tikzcd}
V - x \ar[r,"{\mathrm{tr}_{y-x}}"] \ar[d,"{\zeta_{T,t}}"] & V - y \ar[d,"{\zeta_{T,t}}"] \\
U - a \ar[r,"{\mathrm{tr}_{b-a}}"] & U - b 
\end{tikzcd}
\]
of complex analytic isomorphisms. Via this diagram, the composite map above is canonically identified with 
\[
\begin{array}{rcl}
H_T(X^a) \otimes_{H_T} \O_{E_{T,t}}(V - x) &\cong & H_T(X^b) \otimes_{H_T}  \O_{E_{T,t}}(U - a) \\
&\cong& H_{T/T(a,b)}(X^b) \otimes_{H_{T/T(a,b)}} \O_{E_{T,t}}(U - a) \\ 
&\cong& H_{T/T(a,b)}(X^b) \otimes_{H_{T/T(a,b)}} \O_{E_{T,t}}(U - b) \\
&\cong& H_T(X^b) \otimes_{H_T}  \O_{E_{T,t}}(U - b),
\end{array}
\]
where the fourth map is $\id \otimes \mathrm{tr}^*_{b-a}$. This is the gluing map $\phi_{b,a}$ of Grojnowski's construction (see Remark \ref{perry}), and in Theorem \ref{duhh} we show that it is identical to the gluing map associated to Theorem \ref{iso}.
\end{remark}

\begin{lemma}\label{gart}
With the hypotheses of Remark \ref{gluh}, the translation map
\[
\mathrm{tr}^*_{y -x}: \O_{\t_\C}(V - x) \longrightarrow \O_{\t_\C}(V - y)
\]
is $H_{T/T(a,b)}$-linear.
\end{lemma} 

\begin{proof}
Let $x = x_1t_1 + x_2t_2$ and $y = y_1t_1 + y_2t_2$. By the same argument as in the proof of Lemma \ref{hard}, we have that 
\[
(y_1-x_1, y_2-x_2) \in \mathrm{Lie}(T(a,b)) \times \mathrm{Lie}(T(a,b)),
\]
since $T(a),T(b) \subset T(a,b)$. Therefore
\[
y - x = (y_1-x_1)t_1 + (y_2-x_2)t_2 \in \mathrm{Lie}(T(a,b)) \otimes_\R (\R t_1 + \R t_2) = \mathrm{Lie}(T(a,b)) \otimes_\R \C.
\]
This implies the result.
\end{proof}

\begin{theorem}\label{duhh}
With the hypotheses of Remark \ref{gluh}, the gluing map associated to the local description in Theorem \ref{iso} on an open subset $U \subset U_{a} \cap U_{b}$ is identical to the composite map in Remark \eqref{gluh}.
\end{theorem}

\begin{proof}
We must first make a few observations before we can make sense of the diagram which we will use to prove the theorem. Let $T(x,y)$ be the subgroup of $\bT^2 \times T$ generated by $T(t,x)$ and $T(t,y)$. We have $T(a,b) \subset T \cap T(x,y)$, since
\begin{gather*}
T(a,b) = \langle T(a),T(b) \rangle = \langle T \cap T(t,x), T \cap T(t,y) \rangle \\ \subset T \cap \langle T(t,x),T(t,y) \rangle = T \cap T(x,y).
\end{gather*}
Furthermore, the inclusion $T \subset \bT^2 \times T$ induces an isomorphism $T/(T\cap T(x,y)) \cong (\bT^2 \times T)/T(x,y)$, as may be verified by chasing a diagram analogous to \eqref{rect}. We therefore have identifications 
\begin{gather*}
X^{T \cap T(x,y)} = \Map_T(T/(T\cap T(x,y)),X) \cong \Map_T((\bT^2 \times T)/T(x,y),X) \\ = (L^2X)^{T(x,y)} = L^2X^{t,y} \cong X^b, 
\end{gather*}
where the mapping spaces $\Map_T(-,-)$ of $T$-equivariant maps are identified with fixed-point spaces by evaluating at $0 \in T$. The composite is $T$-equivariant if we let $T$ act on mapping spaces via the target space. Thus, $X^b$ is fixed by $T\cap T(x,y)$, and the homeomorphism $X^b \cong LX^{t,y}$ is equivariant with respect to $T/(T\cap T(x,y)) \cong (\bT^2 \times T)/T(x,y)$. Finally, note that Lemma \ref{hard} implies that 
\begin{equation}\label{dia2}
\begin{tikzcd}
L^2X^{t,y} \ar[r,hook,"{i_{y,x}}"] \ar[d,"{\cong}"',"{ev_{t,y}}"] & L^2X^{t,x} \ar[d,"{\cong}"',"{ev_{t,x}}"]  \\
X^b \ar[r,hook,"{i_{b,a}}"] & X^a
\end{tikzcd}
\end{equation}
commutes. Now, the diagram is as follows.

\begin{equation}\label{dirt}
\begin{tikzcd}[row sep=large, column sep=0.5cm]
H_T(X^a)_{V-x} \ar[r] \ar[d] \drar[phantom, "(1)"] & H_{T/T(a)}(X^a)_{V-x} \ar[r] \ar[d] \drar[phantom, "(2)"] & H_{K(t,x)}(L^2X^{t,x})_{\{t\} \times V} \ar[r] \ar[d] \drar[phantom, "(1)"] & H_{\bT^2\times T}(L^2X^{t,x})_{\{t\} \times V} \ar[d] \\
H_T(X^b)_{V-x} \ar[r] \ar[d]\drar[phantom, "(3)"] & H_{T/T(a)}(X^b)_{V-x} \ar[r] \ar[d] \drar[phantom, "(4)"]& H_{K(t,x)}(L^2X^{t,y})_{\{t\} \times V} \ar[r] \ar[d]  \drar[phantom, "(3)"]& H_{\bT^2\times T}(L^2X^{t,y})_{\{t\} \times V} \ar[d] \\
H_{T/T(a,b)}(X^b)_{V-x} \ar[r] \ar[d] \drar[phantom, "(5)"]& H_{T/(T\cap T(x,y))}(X^b)_{V-x} \ar[r] \ar[d] \drar[phantom, "(6)"]& H_{(\bT^2 \times T)/T(x,y)}(L^2X^{t,y})_{\{t\} \times V} \ar[r] \ar[dl, bend left=15] \dar[phantom, "(4)"] & H_{K(t,y)}(L^2X^{t,y})_{\{t\} \times V} \ar[dl] \\
H_{T/T(a,b)}(X^b)_{V-y} \ar[r] \ar[d] & H_{T/(T\cap T(x,y))}(X^b)_{V-y} \dlar[phantom, "(3)"'] & H_{T/T(b)}(X^b)_{V-y} \ar[dll] \ar[l] \\
H_T(X^b)_{V-y} \\
\end{tikzcd}
\end{equation}

Each map is an isomorphism of $\O_{\t_\C}(V)$-algebras, and is exactly one of the following four types:
\begin{itemize}
\item the change of groups map of Proposition \ref{changeh} (if the target and source only differ by equivariance group);
\item the map induced by an inclusion of spaces (if the target and source only differ by the topological space); 
\item the translation map $\id \otimes \mathrm{tr}^*_{y-x}$ of Remark \ref{gluh} (these are the vertical maps of region (5) - note that translation by $y-x$ is $H_{T/(T\cap T(x,y))}$-linear, since $T(a,b) \subset T \cap T(x,y)$); or
\item the map of Lemma \ref{equal1}, or a map obtained in a way exactly analogous to the proof of Lemma \ref{equal1} (see below).
\end{itemize}

If diagram \eqref{dirt} commutes, then the two outermost paths from $H_T(X^a)_{V-x}$ to $H_T(X^b)_{V-y}$ are equal, which yields the result. Indeed, this is true, because each of the numbered regions in diagram \eqref{dirt} commutes for the reasons stated respectively below.
\begin{enumerate}[(1)]
\item By naturality of the isomorphism of Proposition \ref{changeh}.
\item By naturality of the isomorphism of Lemma \ref{equal1}, along with diagram \eqref{dia2}.
\item By Lemma \ref{kurt}.
\item This holds essentially because the homeomorphism $X^b \cong LX^{t,y}$ is equivariant with respect to the inclusion $T \subset \bT^2 \times T$. (See below for a more detailed proof.)
\item This is obvious.
\item Note that both $(t,x)$ and $(t,y)$ are contained in the complexified Lie algebra of $T(x,y)$, and are therefore also in the kernel of 
\[
\C^2 \times \t_\C \twoheadrightarrow \mathrm{Lie}((\bT^2 \times T)/T(x,y))_\C.
\]
Thus, in the proof of Lemma \ref{equal1}, we may translate by either $(t,x)$ or $(t,y)$, leading to the horizontal and diagonal maps, respectively. These two maps evidently commute with the vertical map, which is $\id \otimes \mathrm{tr}^*_{y-x}$.
\end{enumerate}

In more detail, the claim of item 4 holds by a proof similar to that of Lemma \ref{equal1}. Indeed, the identification $X^b \cong LX^{t,y}$ is equivariant with respect to the isomorphisms $T/(T\cap T(x,y)) \cong (\bT^2 \times T)/T(x,y)$, $T/T(a) \cong K(t,x)$ and $T/T(b) \cong K(t,y)$, which are all induced by the inclusion $T \subset \bT^2 \times T$. This implies that, in the case of the middle square, we have an isomorphism of the diagram 
\begin{equation}
\begin{tikzcd}
E(T/T(a))  \times_{T/T(a)} X^{b} \ar[d] \ar[r] & B(T/T(a)) \ar[d] \\
E(T/(T\cap T(x,y))) \times_{T/(T\cap T(x,y))} X^b \ar[r] &  B(T/(T\cap T(x,y))), \\
\end{tikzcd}
\end{equation}
which induces the left vertical map, and the diagram  
\begin{equation}
\begin{tikzcd}
E(K(t,x)) \times_{K(t,x)} L^2X^{t,y} \ar[d] \ar[r] & B(K(t,x)) \ar[d] \\
E((\bT^2 \times T)/T(x,y)) \times_{(\bT^2 \times T)/T(x,y)} L^2X^{t,y} \ar[r] &  B((\bT^2 \times T)/T(x,y)), \\
\end{tikzcd}
\end{equation}
which induces the right vertical map. We therefore have a commutative square 
\[
\begin{tikzcd}
H_{T/T(a)}(X^b) \ar[r] \ar[d] & H_{K(t,x)}(L^2X^{t,y}) \ar[d] \\
H_{T/(T\cap T(x,y))}(X^b) \otimes_{H_{T/(T\cap T(x,y))}} H_{T/T(a)} \ar[r] & H_{(\bT^2 \times T)/T(x,y)}(L^2X^{t,y}) \otimes_{H_{(\bT^2 \times T)/T(x,y)}} H_{K(t,x)}
\end{tikzcd}
\]
of isomorphisms of $H_{T/T(a)} \cong H_{K(t,x)}$-algebras. Then, to get isomorphisms of $\O_{\t_\C}(V)$-algebras, we tensor the left hand side over $p_a \circ \mathrm{tr}_{-x}$ and the right hand side over $p_{t,x}\circ \iota_t$, as in the diagram of Lemma \ref{dg}. One shows the commutativity of the other square labelled (4) in exactly the same way, replacing $T(a)$ with $T(b)$ and $K(t,x)$ with $K(t,y)$.
\end{proof}

\begin{corollary}\label{glutes}
There is an isomorphism of reduced $T$-equivariant elliptic cohomology theories 
\[
\widetilde{\E}^*_{T,t} \cong \widetilde{\G}^*_{T,t},
\]
in the sense of Definition \ref{tear}.
\end{corollary}

\begin{proof}
The local description of $\E_{T,t}(X)$ given in Theorems \ref{iso} and \ref{duhh} amounts to the natural isomorphism that we require. It is clear that this is compatible with suspension isomorphisms since, in each theory, these are induced by the suspension isomorphism on cohomology, by the proofs of Propositions \ref{kerry} and \ref{jezza}. 
\end{proof}

\end{document}